\documentclass[a4paper,11pt]{article}
\usepackage[utf8]{inputenc}
\usepackage[T1]{fontenc}
\usepackage[english]{babel}
\usepackage{amsmath,amssymb,amsthm}
\usepackage[short]{optidef}				
\usepackage[shortlabels]{enumitem} 
\usepackage{comment}
\usepackage{standalone}
\usepackage{mathtools}
\usepackage{fullpage}

\usepackage{float}
\usepackage{mathdots}
\usepackage{tikz}
\usepackage[subrefformat=parens,labelformat=parens]{subcaption}
\graphicspath{{./Figuras/}} 
\usetikzlibrary{graphs,graphs.standard}
\usepackage{hyperref}
\usepackage{thmtools}
\usepackage{thm-restate}
\tikzstyle{line} = [draw, thick]
\tikzset{vertex/.style={circle,draw,fill=black,inner sep=1.5pt}}

\usetikzlibrary{decorations.pathmorphing,decorations.pathreplacing,calc}

\usetikzlibrary{decorations.pathmorphing}
\usetikzlibrary{positioning,chains,fit,shapes,calc}

\tikzstyle{selected edge} = [draw,line width=5pt,-,red!50]
\tikzstyle{vertex}=[circle,draw,fill=white,minimum size=10pt,inner sep=2pt]
\tikzstyle{black vertex}=[circle,draw,fill=black,minimum size=10pt,inner sep=2pt]
\tikzstyle{red vertex}=[circle,draw,fill=red!90,minimum size=10pt,inner sep=2pt]
\tikzstyle{gray vertex}=[circle,draw,fill=gray!90,minimum size=20pt,inner sep=1pt]
\tikzstyle{trivial}=[circle,draw,fill=black,minimum size=7pt,inner sep=1pt]
\tikzstyle{dummy}=[circle,minimum size=5pt,inner sep=0pt]
\tikzstyle{red edge} = [draw,very thick,-,red!80,decorate,decoration={snake, amplitude=.8pt, segment length=4pt}]
\tikzstyle{purple edge} = [draw,very thick,-,magenta]
\tikzstyle{blue edge} = [draw,very thick,-,blue!60]
\tikzstyle{green edge} = [draw,very thick,-,green!90]
\tikzstyle{brown edge} = [draw,thick,-,brown!90]
\tikzstyle{black edge} = [draw,thick,-,black!90]
\tikzstyle{dot} = [circle,inner sep=1pt,fill,name=#1]
\tikzstyle{extended line} = [shorten >=-#1,shorten <=-#1]

\definecolor{myblue}{RGB}{80,80,160}
\definecolor{mygreen}{RGB}{80,160,80}
\definecolor{myred}{RGB}{255,0,0}
\definecolor{mybrown}{RGB}{139,69,19}

\newtheorem{theorem}             {Theorem}
\newtheorem{lemma}     	[theorem] {Lemma}        
\newtheorem{conjecture}	[theorem] {Conjecture}

\newtheorem{proposition}[theorem] {Proposition}

\newtheorem{corollary}	[theorem] {Corollary}
     
\newtheorem{claim}{Claim}

\newtheoremstyle{case}{}{}{}{}{\bfseries}{:}{ }{}
\theoremstyle{case}

\numberwithin{subcase}{case}

\begin{document}

\title{Separating Matchings in Cubic Graphs}

\author{
Renzo Gómez\thanks{Center of Mathematics, Computation and Cognition, Federal University of ABC, Brazil. E-mail: {\tt gomez.renzo@ufabc.edu.br}}
\and
Juan Gutierrez\thanks{Faculty of Computing, University of Engineering and Technology, Peru. E-mail: {\tt jgutierreza@utec.edu.pe}. This work was partially supported by Fondo Semilla UTEC 2025.}
}
\date{}
\maketitle

\begin{abstract}
We study separating matchings in graphs, that is, matchings whose removal increases the number of connected components, and focus on determining the maximum size of such a matching in a graph $G$, denoted by $\mathrm{mms}(G)$. We show that every subcubic graph admits a separating matching, except for exactly eight graphs, which allows us to focus on bounding $\mathrm{mms}(G)$ for cubic graphs. Our main results show that every cubic graph $G$ on $n$ vertices that admits a separating matching satisfies $\mathrm{mms}(G) \ge n/2 - 2$. For bipartite cubic graphs, assuming a conjecture of Funk, the problem reduces to a recursively defined class $\mathcal{F}$, for which we prove that $\mathrm{mms}(G) \ge n/2 - 1$, up to four exceptional graphs. In contrast, we show that every claw-free cubic graph satisfies $\mathrm{mms}(G) = n/2$. These results extend previous work on matching cuts and disconnected $2$-factors, and provide the first systematic study of maximum separating matchings.
\end{abstract}

\section{Introduction} \label{sec:intro}

In this paper, all graphs are simple, and we use standard terminology and notation \cite{Bondy2008,Diestel18}. A classical result due to Petersen states that every cubic bridgeless graph has a perfect matching, and hence a $2$-factor. A natural question is whether such a graph admits a disconnected $2$-factor, that is, a $2$-factor that is not a Hamilton cycle. Diwan~\cite{Diwan02} proved that every planar cubic bridgeless graph on at least six vertices has a disconnected $2$-factor. Equivalently, every such graph admits a \emph{disconnecting perfect matching}, that is, a perfect matching whose removal increases the number of connected components.

We study an optimization version of this problem: determine the maximum size of a \emph{separating matching} in a graph, that is, a matching whose removal increases the number of connected components. We denote this parameter by $\mathrm{mms}(G)$, and if $G$ has no separating matching, we set $\mathrm{mms}(G)=0$. Although this problem has not been studied explicitly, the related notion of a matching cut has received considerable attention \cite{Bouquet2025,le2023maximizingmatchingcuts,Le2019,Le2023,LuckePaulusmaRies2022,LuckePaulusmaRies2023,Patrignani2001}. A \emph{matching cut} is a matching that is also an edge cut. While not every graph admits such a structure, the existence of a separating matching is equivalent to the existence of a matching cut; graphs with this property are called \emph{decomposable}.

The study of decomposable graphs was initiated by Graham~\cite{Graham1970}. Chv\'atal~\cite{Chvatal1984} showed that recognizing decomposable graphs is polynomial-time solvable for subcubic graphs, but NP-complete for graphs of maximum degree~$4$. The problem is also polynomial-time solvable for graphs of radius at most~$2$ \cite{LuckePaulusmaRies2022}, and NP-complete for graphs of diameter~$3$ \cite{Le2019}. For bipartite graphs, it is polynomial-time solvable when the diameter is at most~$3$, and NP-complete when the diameter is~$4$ \cite{Le2019}. Moshi~\cite{Moshi89} proved that every subcubic graph on at least $8$ vertices is decomposable. We strengthen this result as follows.

\begin{restatable}{theorem}{nondecomp}\label{thm:nondecomp}
Every subcubic graph is decomposable, except for the eight graphs in Figure~\ref{fig:D}.
\end{restatable}

This allows us to focus on bounding $\mathrm{mms}(G)$ for cubic graphs. In particular, Diwan's result can be restated as follows.

\begin{theorem}[{\normalfont Diwan~\cite{Diwan02}}]
If $G$ is a planar cubic bridgeless graph on at least $6$ vertices, then $\mathrm{mms}(G)=n/2$.
\end{theorem}

Let $\nu(G)$ denote the size of a maximum matching in $G$. If a cubic graph contains a bridge, we can bound $\mathrm{mms}(G)$ as follows.

\begin{theorem}\label{th:cubic-bridge}
If $G$ is a decomposable cubic graph with a bridge, then $\mathrm{mms}(G)\geq \nu(G) - 1$.
\end{theorem}

\begin{proof}
Let $M$ be a maximum matching in $G$, and let $e$ be a bridge of $G$. If $e \in M$, then $M$ separates $G$. Otherwise, write $e=uv$. Since $M$ is maximum, at least one edge of $M$ is incident to $u$ or $v$. If exactly one such edge $e'$ exists, define $M'=(M \setminus \{e'\}) \cup \{e\}$. If there are two such edges $e'$ and $e''$, define $M'=(M \setminus \{e',e''\}) \cup \{e\}$. In both cases, $M'$ is a separating matching and $|M'| \geq |M|-1$.
\end{proof}

Restricting to $2$-edge-connected cubic graphs, we obtain the following result.

\begin{restatable}{theorem}{cubic}\label{th:cubicn2-2}
For every $2$-edge-connected decomposable cubic graph $G$ on $n$ vertices, $\mathrm{mms}(G)\geq n/2 - 2$.
\end{restatable}

For bipartite cubic graphs, a conjecture of Funk~\cite{Funk2003} states that the class of graphs that do not admit a disconnecting perfect matching coincides with a recursively defined class $\mathcal{F}$. Thus, under this conjecture, bounding $\mathrm{mms}(G)$ for bipartite cubic graphs reduces to understanding graphs in $\mathcal{F}$. We prove the following.

\begin{restatable}{theorem}{bicubic}\label{th:bicubicn2-1}
For every graph $G \in \mathcal{F}$ on $n$ vertices, $\mathrm{mms}(G)\geq n/2 - 1$, except for the four graphs in Figure~\ref{fig:F}.
\end{restatable}

Our final result concerns claw-free cubic graphs. It is known that every connected claw-free graph with an even number of vertices has a perfect matching \cite{LasVergnas1975,Sumner1974}, and in particular this holds for connected claw-free cubic graphs. We strengthen this by showing that such a graph always admits a disconnecting perfect matching.

\begin{restatable}{theorem}{clawfree}\label{th:clawfree-n2}
Every connected claw-free cubic graph distinct from $K_4$ has a disconnecting perfect matching.
\end{restatable}

\section{Decomposable subcubic graphs} \label{sec:nondecomp}

In this section, we focus on characterizing which subcubic graphs are decomposable. 
Moshi~\cite{Moshi89} showed the following result regarding decomposable  
cubic graphs.
\begin{theorem}[Moshi, 1989] \label{thm:moshi}
    If $G$ is a connected cubic graph that is not isomorphic to $K_4$ 
    or $K_{3,3}$, then $H$ is decomposable. 
\end{theorem}
Furthermore, he argued that every subcubic graph with at least $8$ vertices is decomposable. 
Let $\overline{\mathcal{D}}$ denote the family of subcubic graphs that are not decomposable. 
Note that $K_3$ is the only cycle that belongs to $\overline{\mathcal{D}}$. 
Now, let $G \in \overline{\mathcal{D}}$ such that $G \neq K_3, K_4, K_{3,3}$. 
Our objective will be to characterize the structure of such a graph $G$. First, 
note that $G$ is $2$-edge-connected and, since $G \neq K_3$, we have that $\Delta(G) = 3$. 
Also, by Theorem~\ref{thm:moshi}, we have that $\delta(G) = 2$. 

We call a path, in $G$, \textit{subdivided} if it contains at least two edges, 
the ends of the path have degree at least three in $G$, and every internal vertex 
in the path have degree two in $G$. Observe that every subdivided path in $G$ has 
exactly two edges, otherwise, $G$ has a matching cut. Thus,~$G$ satisfies the following 
property: 
\begin{enumerate}[$(i)$]
    \item every vertex of degree two is adjacent only to vertices of degree three in $G$. \label{it:a}
\end{enumerate}

In what follows, we relate the study of nondecomposable subcubic graphs to 
the study of nondecomposable cubic multigraphs (graphs that may contain parallel 
edges and loops). In this regard, given a subcubic graph $G$ which is not a cycle, 
if we replace every subdivided path with an edge, then we obtain a cubic multigraph. 
Let $\mathcal{C}(G)$ be the graph obtained from this 
procedure. Our first result states the following. 
\begin{proposition} \label{prop:cont}
    Let $G$ be a subcubic graph. If $\mathcal{C}(G)$ has 
    a matching separator, then $G$ has a matching cut.  
\end{proposition}
\begin{proof}
    Let $M$ be a matching separator of $\mathcal{C}(G)$. Each edge in $M$ 
    is an edge of $G$ or is a subdivided path. To obtain a matching separator of $G$, we choose 
    the edges of $M$ that also belong to $E(G)$ and one edge from each path that corresponds 
    to an edge in $M$. Since the internal vertices of those paths have degree two, we obtain 
    a matching separator. 
\end{proof}

The previous result tells us that to find the nondecomposable subcubic graphs (different 
from $K_3$), it is important to discover the nondecomposable cubic multigraphs first. In 
what follows, we denote by $S_2$ the unique cubic multigraph with two vertices. 

\begin{figure}
    \centering
    \includegraphics[width=\linewidth]{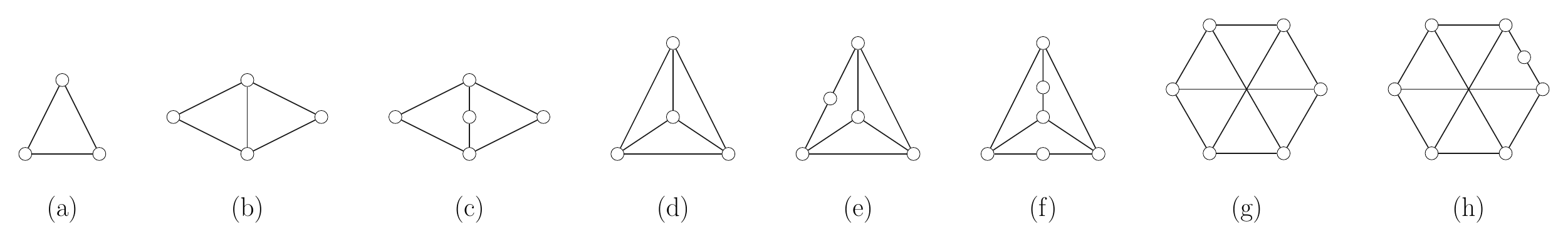}
    \caption{Family $\overline{\mathcal{D}}$ of nondecomposable subcubic graphs: (a)~$\overline{D}_0$;
    (b)~$\overline{D}_1$; (c)~$\overline{D}_2$; (d)~$\overline{D}_3$; (e)~$\overline{D}_4$; (f)~$\overline{D}_5$; 
    (g)~$\overline{D}_6$; and (h)~$\overline{D}_7$.}
    \label{fig:D}
\end{figure}

\begin{theorem} \label{thm:multi}
    Let $H$ be a $2$-edge-connected cubic multigraph. If $H \notin \{ S_2, K_4, K_{3,3} \}$, 
    then $H$ is decomposable. 
\end{theorem}
\begin{proof}
    First, note that $H$ has no loops since it is $2$-edge-connected. Moreover,  
    by Petersen's Theorem~\cite{Petersen1900}, $H$ has a perfect matching, say $M$. If $H - M$ 
    is disconnected, then $H$ is decomposable. Now, suppose that this is not the case and 
    let $C := H - M$ be a Hamiltonian cycle. Observe that $H$ has at least $4$ vertices, 
    otherwise, it implies that $H$ is isomorphic to~$S_2$. In what follows, we distinguish 
    two cases depending on whether $H$ has a parallel edge or not. 

    First, if $H$ has no parallel edges, then $H$ is a simple graph. Thus, by Theorem~\ref{thm:moshi}, 
    $H$ is decomposable. Now, consider that $H$ has a multiple edge, say $uv$. Observe that, 
    since $|C| \geq 4$, the edges incident to $u$ and $v$ different from $uv$ form 
    a matching cut of $H$. Therefore, $H$ is decomposable.
\end{proof}

Finally, by Theorem~\ref{thm:multi}, Proposition~\ref{prop:cont} and property~\ref{it:a}, 
if $G \in \overline{\mathcal{D}}$ and $G \neq K_3, K_4, K_{3,3}$, then $G$ arises from a nondecomposable 
cubic multigraph in which we replace, at most once, each edge by a path of length two. This 
already implies that there are finitely many nondecomposable subcubic graphs. 
The last result of this section shows the complete family of nondecomposable subcubic graphs. 

\nondecomp*
\begin{proof}
    By Theorem~\ref{thm:multi}, we only need to consider the simple graphs that 
    arise from a subdivisions of $S_3$, $K_4$, or $K_{3,3}$. First, consider 
    that $G$ is a graph obtained by subdividing, at most once, each of the edges 
    in $S_3$. Since $S_3$ only contains three parallel edges, $G$ is obtained by 
    subdividing either two or three of those edges. In what follows, we show that 
    in either case, $G$ is nondecomposable. Let $u$ and $v$ be the vertices of 
    degree $3$ in $G$, and let $M$ be any matching in $G$. As there is at most one 
    edge of $M$ incident to $u$ or $v$, and $G$ contains three edge-disjoint paths 
    between $u$ and $v$, we have that these vertices are connected in $G - M$. Moreover, 
    each of the other vertices of $G$ have degree $2$, thus, each one is adjacent either to $u$ 
    or $v$ in $G - M$. Therefore, $G - M$ is connected for any matching $M$ in $G$. 
    Observe that if we subdivide two edges in $S_2$, we obtain $K_4 - e = \overline{D}_1$, 
    and if we subdivide three edges, we obtain $K_{2,3} = \overline{D}_2$.

    Before we continue, we show a claim that will be useful in the rest of the proof.
    \begin{claim} \label{cl:cut}
        Let $H$ be a cubic graph. If $G$ is a graph obtained from $H$ by subdividing at 
        least two edges incident to a common vertex, then $G$ has a matching cut.
    \end{claim}
    \begin{proof}
        Let $G$ be a graph obtained by subdividing at least to edges from $H$, and let  
        $uv$ and $uw$ in $E(H)$ be two of those edges. Consider that   
        $x$ and $y$ are the vertices that arise from these subdivisions such that $ux$, $xv$, 
        $uy$ and $yw$ belong to $E(G)$. Moreover, let $z$ be the neighbor of $u$ in $G$ different 
        from $x$ and $y$. Note that, $M = \{uz, xv, yw \}$ is a matching that separates $u$ in~$G$.
    \end{proof}
    Now, consider that $G$ is a graph obtained by subdividing, at most once, each of the edges 
    in~$K_4$. Since a maximum matching of $K_4$ has size $2$, any subset of at least $3$ edges will 
    have two edges incident to a common vertex. Thus, by Claim~\ref{cl:cut}, the resulting graph 
    will have a matching cut. Therefore, we need to consider only two cases: (a) $G$ is obtained by 
    subdividing a single edge; and (b) $G$ is obtained by subdividing two edges that form a matching 
    in~$K_4$. In what follows $X$ denotes the set of vertices of degree $3$ in $G$.

    First, suppose that $G$ is obtained by subdividing a single edge, and let $M$ be a matching 
    in~$G$. Observe that, this case corresponds to $G = \overline{D}_5$. We will show that~$X$ is 
    contained in a connected component of $G - M$. 
    Note that this implies that $G - M$ is connected. Since the unique vertex of degree $2$ in $G$ has 
    a neighbor in $X$ when we consider $G - M$. Let $u$ and $v$ be distinct vertices in $X$. Observe 
    that, in $G - M$, these vertices have at least two neighbors. Thus, the components that contain $u$ 
    and $v$ have at least $3$ vertices. Since $G$ has $5$ vertices, we have that~$u$ and~$v$ belong to the 
    same connected component in $G - M$. As $u$ and $v$ are arbitrary vertices in $X$, we 
    have that $G$ is nondecomposable. 

    Now, suppose that $G$ is obtained by subdividing two edges that form a matching in $K_4$, and let 
    $M$ be a matching in $G$. Note that, in this case $G = \overline{D}_5$. As in the previous case, we 
    will show that $X$ is contained in a connected component of $G - M$. Let $u \in X$. Since we subdivided 
    the edges of a matching of size two in $K_4$, 
    we have that, in $G$, the vertex $u$ is adjacent to one vertex of degree $2$, say $z$, and two vertices 
    of degree $3$, say $x$ and $y$. Furthermore, as $M$ is a matching, $u$ is adjacent to two vertices 
    in $\{x, y, z\}$. Without loss of generality, suppose that $x$ and $z$ are neighbors of $u$ in $G - M$ 
    (the other cases are symmetric). Note that $xz \notin E(G)$, otherwise, $ux$ is a parallel edge in $K_4$, 
    a contradiction. Thus, as $x$ has degree three, it is adjacent to a vertex different from $u$ 
    and $z$ in $G - M$. Therefore, the connected component in $G - M$ that contains $u$ has at least $4$ 
    vertices. Since $G$ has $6$ vertices, it implies that $X$ is inside a connected component of $G$. 

    Finally, suppose that $G$ is a graph obtained by subdividing, at most once, each of the edges 
    in~$K_{3,3}$. Consider that $U = \{x, y, z\}$ and $W = \{u, v, w\}$ is the bipartition 
    of $K_{3,3}$. First, we will show that if we subdivide at least two edges to obtain $G$, then 
    $G$ is decomposable. By Claim~\ref{cl:cut}, we may suppose that these edges form a matching in $K_{3,3}$.
    Let $e, f \in E(K_{3,3})$ be two of those edges. Without loss of generality, suppose that $e = ux$ 
    and $f = vy$ (the other cases are analogous), and let $s$ and $t$ be the vertices
    that resulted from the subdivision of $e$ and $f$, respectively. Then, 
    the matching $M = \{us, wx, ty, vz\}$ separates $G$ in exactly two components, one of which 
    contains the vertices $\{s,x,v,t\}$. 

    The only case left to consider is when~$G$ is obtained by subdividing a single edge in $K_{3,3}$. 
    Thus, $G = \overline{D}_7.$ We will show that the vertices in $U \cup W$ are contained in a connected 
    component of $G - M$, for any matching $M$ in $G$. Let $a \in U$. Since $a$ has degree $3$ in $G$, we have 
    that $a$ has two neighbors, say $b$ and $c$, in $G - M$. Moreover, as $G$ results from the subdivision 
    of one edge, we have that $b$ or $c$ (or both) belongs to $W$ and, thus, also has degree $3$. 
    Without loss of generality, suppose that $b \in W$. 
    Note that, if $c$ has degree $2$ in $G$, then $bc \notin E(G)$, otherwise, $bc$ would be 
    a parallel edge in $K_{3,3}$, a contradiction. Furthermore, if $c$ has degree $3$ in $G$, 
    then $c \in W$. In this case also $bc \notin E(G)$. Thus, $b$ has a neighbor in $G - M$ 
    that is different from $u$ and $c$. This implies that the connected component that contains $a$ 
    has at least $4$ vertices. A symmetric argument shows the same for every vertex of $W$. Therefore, 
    the vertices in $U \cup W$ belong to the same connected component of $G - M$, so $G$ is not decomposable. 
\end{proof}

\section{Matching separators in cubic graphs} \label{sec:cubic}

Given a matching~$M$ in a graph~$G$, we say that a path  
$P=\langle u_1, u_2, \ldots, u_k \rangle$ is an~$M$-\textit{reducing path}  
if $k$ is even, $u_i u_{i+1} \in M$ for every odd~$i$, and  
$u_i u_{i+1} \notin M$ for every even~$i$.  
We define
\[
M \otimes P = \bigl(M \setminus \{u_i u_{i+1} : i \text{ odd}\}\bigr) 
\cup \{u_i u_{i+1} : i \text{ even}\}.
\]

\cubic*

\begin{proof}
By Petersen's Theorem~\cite{Petersen1900}, $G$ has a perfect matching~$M$.  
If $G - M$ is disconnected, then $M$ is a separating matching, and we are done.  
Hence, we may assume that $G - M$ is connected.  
Since $G - M$ is $2$-regular, it contains a Hamiltonian cycle, say~$C$.

Let $C = \langle v_0, v_1, \ldots, v_{n-1} \rangle$.  
For $i,j \in \{0,\ldots,n-1\}$, we denote by $C[i,j]$ the $(v_i,v_j)$-subpath of $C$ obtained by traversing $C$ from $v_i$ to $v_j$ in the forward direction.  
We say that an edge of $G$ is a \textit{chord} if it joins two nonadjacent vertices of $C$.  
Two chords \emph{cross} if their endpoints appear in alternating order along $C$, and are \textit{parallel} otherwise.  
For each vertex $v_i$, let $p_i$ denote the unique neighbour of $v_i$ in $G$ that is not adjacent to $v_i$ on $C$.  
A \textit{tie} is a pair of crossing chords that induce a $C_4$ in $G$.  
Throughout this section, indices are taken modulo $n$.

We divide the proof into two cases, depending on whether $G$ contains a tie or not.  
We will use the following claim repeatedly.

\begin{claim}
\label{clm:reducing-path}
If $v_a v_b$ is a chord of $C$, then the path 
$\langle p_a, v_a, v_b, p_b \rangle$ is an $M$-reducing path. 
Moreover, if two such paths are edge-disjoint, then their successive application 
produces a matching of size $|M|-2$.
\end{claim}

\vspace{3mm}
\noindent \textbf{Case 1:} $G$ contains a tie. \\

Let $(v_i v_j, v_{i+1} v_{j+1})$ be a tie in $G$.
If both $|C[i+1,j]|$ and $|C[j+1,i]|$ are less than $3$, then $G$ is either $K_4$ or $K_{3,3}$, 
a contradiction. Thus, at least one of them has size at least $3$. 
Without loss of generality, assume that $|C[j+1,i]| \ge 3$.

If $|C[i+1,j]| \ge 3$, then applying Claim~\ref{clm:reducing-path} to the chords 
$v_{i+1}v_{j+1}$ and $v_i v_j$, we obtain
\[
M' = (M \otimes \langle p_{i+2}, v_{i+2}, v_{i+1}, v_{j+1}, v_{j+2}, p_{j+2} \rangle)
\otimes \langle p_{j-1}, v_{j-1}, v_j, v_i, v_{i-1}, p_{i-1} \rangle,
\]
where $|M'| = |M| - 2$. 
Moreover, the vertices $\{v_i, v_{i+1}, v_j, v_{j+1}\}$ lie in a component of $G - M'$ 
that does not contain $v_{i-1}$ (see Figure~\ref{fig:cubic-cases12}(a)), and we are done.

If $|C[i+1,j]| \le 2$ and $p_{i-1}, p_{j+2} \in V(C[j+1,i])$, then applying 
Claim~\ref{clm:reducing-path} yields
\[
M' = (M \otimes \langle p_{i-1}, v_{i-1}, v_i, v_j \rangle)
\otimes \langle p_{j+2}, v_{j+2}, v_{j+1}, v_{i+1} \rangle,
\]
where $|M'| = |M| - 2$. 
In this case, the vertices $\{v_i, v_{i+1}, v_j, v_{j+1}\}$ lie in a component of $G - M'$ 
that does not contain $v_{i-1}$ (see Figure~\ref{fig:cubic-cases12}(b)), and we are done.

Hence, we may assume that $|C[i+1,j]| = 2$ and, without loss of generality, that $p_{i-1} = v_{i+2}$.
Applying Claim~\ref{clm:reducing-path} again, we obtain
\[
M' = (M \otimes \langle p_{i-2}, v_{i-2}, v_{i-1}, v_{i+2} \rangle)
\otimes \langle p_{j+2}, v_{j+2}, v_{j+1}, v_{i+1} \rangle,
\]
where $|M'| = |M| - 2$. 
Moreover, the vertices $\{v_{i-1}, v_i, v_{i+2}, v_j, v_{j+1}\}$ lie in a component of $G - M'$ 
that does not contain $v_{j+2}$ (see Figure~\ref{fig:cubic-cases34}(a)). 
Therefore, the proof follows.

\begin{figure}[H]
    \centering
    \begin{subfigure}{0.49\textwidth}
        \centering
        \scalebox{0.7}{
            \begin{tikzpicture}
                \draw[black, thick] (0,0) circle (3cm);
                
                \foreach \i/\name/\label/\anchor in {
                100/vi/$v_i$/south,
                80/vi+1/$v_{i+1}$/south,
                60/vi+2/$v_{i+2}$/south,
                120/vi-1/$v_{i-1}$/south,
                280/vj/$v_j$/north,
                300/vj-1/$v_{j-1}$/north,
                260/vj+1/$v_{j+1}$/north,
                240/vj+2/$v_{j+2}$/north} {
                    \node[black, fill, circle, inner sep=3pt] (\name) at (\i:3cm) {};
                    \node[black, anchor=\anchor] at (\i:3.3cm) {\label};
                }
                
                \draw[blue, thick] (vi) -- (vj);
                \draw[blue, thick] (vi+1) -- (vj+1);

      \draw[red, thick,
      decorate,
      decoration={snake, amplitude=0.8mm, segment length=3mm}]
      (vi+1) -- (vi+2);
      \draw[red, thick,
      decorate,
      decoration={snake, amplitude=0.8mm, segment length=3mm}]
      (vj+1) -- (vj+2);
      \draw[red, thick,
      decorate,
      decoration={snake, amplitude=0.8mm, segment length=3mm}]
      (vi-1) -- (vi);
      \draw[red, thick,
      decorate,
      decoration={snake, amplitude=0.8mm, segment length=3mm}]
      (vj) -- (vj-1);

            \end{tikzpicture}
        }
        \caption{}
    \end{subfigure}
    \hfill
    \begin{subfigure}{0.49\textwidth}
        \centering
        \scalebox{0.7}{
            \begin{tikzpicture}
                \draw[black, thick] (0,0) circle (3cm);
                
                \foreach \i/\name/\label/\anchor in {
                100/vi/$v_i$/south,
                80/vi+1/$v_{i+1}$/south,
                120/vi-1/$v_{i-1}$/south,
                280/vj/$v_j$/north,
                260/vj+1/$v_{j+1}$/north,
                240/vj+2/$v_{j+2}$/north,
                160/pj+2/$p_{j+2}$/east,
                200/pi-1/$p_{i-1}$/east} {
                    \node[black, fill, circle, inner sep=3pt] (\name) at (\i:3cm) {};
                    \node[black, anchor=\anchor] at (\i:3.3cm) {\label};
                }
                
                \draw[blue, thick] (vi) -- (vj);
                \draw[blue, thick] (vi+1) -- (vj+1);
                \draw[black, thick] (vi-1) -- (pi-1);
                \draw[black, thick] (vj+2) -- (pj+2);

      \draw[red, thick,
      decorate,
      decoration={snake, amplitude=0.8mm, segment length=3mm}]
      (vj+1) -- (vj+2);
      \draw[red, thick,
      decorate,
      decoration={snake, amplitude=0.8mm, segment length=3mm}]
      (vi-1) -- (vi);

            \end{tikzpicture}
        }
        \caption{}
    \end{subfigure}
    \caption{Situations in the proof of Theorem \ref{th:cubicn2-2} when a tie~$(v_iv_j,v_{i+1}v_{j+1})$ 
    exist in $G$. In~$(a)$,~$|C[i+1,j]|\geq 3$; and in~$(b)$,~$|C[i+1,i]|\leq 2$.}
    \label{fig:cubic-cases12}
\end{figure}
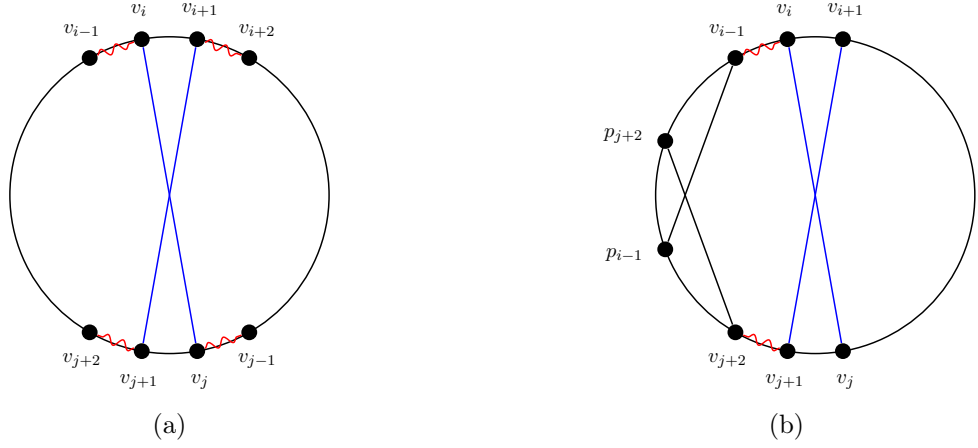


    Hence, from now on, we may assume that at least one of~$p_{i-1}$ or~$p_{j+2}$ is in~$C[i+1,j]$, 
    which implies that~$|C[i+1,j]|=2$. Let us suppose without loss of generality that~$p_{i-1}=v_{i+2}$.
    Let $$M'=(M \otimes p_{i-2}v_{i-2}v_{i-1}v_{i+2}) \otimes p_{j+2}v_{j+2}v_{j+1}v_{i+1}.$$
    Note that the vertices~$\{v_{i-1},v_{i},v_{i+2},v_j,v_{j+1}\}$
    are separated from~$v_{j+2}$ in~$G-M'$ (Figure \ref{fig:cubic-cases34}~$(a)$).
    Thus,~$|M'|\geq |M|-2$ and the proof follows.
\begin{figure}[H]
    \centering
    \begin{subfigure}{0.49\textwidth}
        \centering
        \scalebox{0.7}{
            \begin{tikzpicture}
                \draw[black, thick] (0,0) circle (3cm);
                
                \foreach \i/\name/\label/\anchor in {
                100/vi/{$\Large v_i$}/south,
                80/vi+1/{$\Large v_{i+1}$}/south,
                0/vi+2/{$\Large v_{i+2}$}/west,
                120/vi-1/{$\Large v_{i-1}$}/south,
                140/vi-2/{$\Large v_{i-2}$}/south,
                280/vj/{$\Large v_j$}/north,
                260/vj+1/{$\Large v_{j+1}$}/north,
                240/vj+2/{$\Large v_{j+2}$}/north,
                220/pi-2/{$\Large p_{i-2}$}/north,
                160/pj+2/{$\Large p_{j+2}$}/east} {
                    \node[black, fill, circle, inner sep=3pt] (\name) at (\i:3cm) {};
                    \node[black, anchor=\anchor] at (\i:3.3cm) {\label};
                }
                
                \draw[blue, thick] (vi) -- (vj);
                \draw[blue, thick] (vi+1) -- (vj+1);
                \draw[black, thick] (vi-1) -- (vi+2);
                \draw[black, thick] (vj+2) -- (pj+2);
                \draw[black, thick] (vi-2) -- (pi-2);

                 \draw[red, thick,
      decorate,
      decoration={snake, amplitude=0.8mm, segment length=3mm}]
      (vj+1) -- (vj+2);
      \draw[red, thick,     decorate,  decoration={snake, amplitude=0.8mm, segment length=3mm}]      (vi-1) -- (vi-2);
      
            \end{tikzpicture}
        }
        \caption{}
    \end{subfigure}
    \hfill
    \begin{subfigure}{0.49\textwidth}
        \centering
        \scalebox{0.7}{
            \begin{tikzpicture}
                \draw[black, thick] (0,0) circle (3cm);
                
                \foreach \i/\name/\label/\anchor in {
                100/x/{$\Large v_x$}/south,        
                60/vb/{$\Large v_{b}$}/south,
                15/px/{$\Large p_{x}$}/west,
                120/va/{$\Large v_{a}$}/south,
                -15/py/{$\Large p_y$}/north,
                300/pb/{$\Large p_{b}$}/north,
                260/y/{$\Large v_y$}/north,
                240/pa/{$\Large p_{a}$}/north} {
                    \node[black, fill, circle, inner sep=3pt] (\name) at (\i:3cm) {};
                    \node[black, anchor=\anchor] at (\i:3.3cm) {\label};
                }
                
                \draw[blue, thick] (va) -- (pa);
                \draw[blue, thick] (vb) -- (pb);
                \draw[black, thick] (x) -- (px);
                \draw[black, thick] (y) -- (py);

                 \draw[red, thick,
      decorate,
      decoration={snake, amplitude=0.8mm, segment length=3mm}]
      (va) -- (x);
      \draw[red, thick,
      decorate,
      decoration={snake, amplitude=0.8mm, segment length=3mm}]
      (pa) -- (y);
      
            \end{tikzpicture}
        }
        \caption{}
    \end{subfigure}
    \caption{Situations in the proof of Theorem \ref{th:cubicn2-2}. In~$(a)$, a tie~$(v_iv_j,v_{i+1}v_{j+1})$ 
    exists, $|C[i+1,j]|\leq 2$ and~$p_{i-1}=v_{i+2}$; and in~$(b)$, there exist no ties.}
    \label{fig:cubic-cases34}
\end{figure}
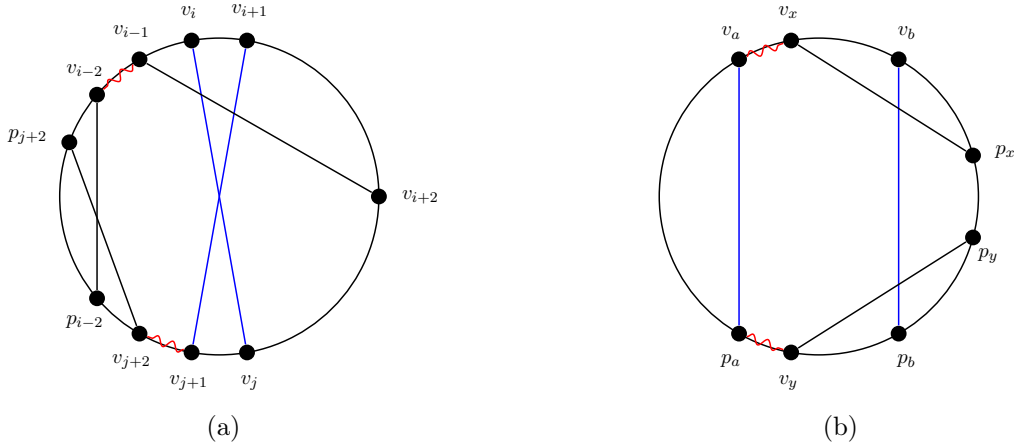

\noindent \textbf{Case 2:} $G$ contains no ties. \\

We first show that $G$ contains two parallel chords. Suppose, for a contradiction, that this is not the case. 
Let $v_i v_{i+1}$ be an edge of $C$ such that the distance between $p_i$ and $p_{i+1}$ along $C$ is maximized. 
Since $G$ contains no ties, the path $C[p_i,p_{i+1}]$ contains an internal vertex, say $v_k$. 
It follows that the chord $v_k p_k$ is parallel to either $v_i p_i$ or $v_{i+1} p_{i+1}$, a contradiction.

Thus, there exist vertices $v_a, v_b$ with $a<b$ such that the chords $v_a p_a$ and $v_b p_b$ are parallel. 
Choose such a pair maximizing $|b-a| + |p_a - p_b|$. Let $v_x$ be the neighbour of $v_a$ in $C[v_a,v_b]$, 
and let $v_y$ be the neighbour of $p_a$ in $C[p_b,p_a]$ (possibly $v_x = v_b$ or $v_y = p_b$). 
Since $G$ contains no ties, both $p_x$ and $p_y$ lie in $V(C[x,y])$; otherwise, we obtain a pair of parallel chords 
contradicting the choice of $v_a$ and $v_b$.

Applying Claim~\ref{clm:reducing-path}, we obtain
\[
M' = M \otimes \langle p_x, v_x, v_a, p_a, v_y, p_y \rangle,
\]
where $|M'| = |M| - 1$. 
Moreover, $v_a$ and $v_b$ lie in distinct components of $G - M'$ 
(see Figure~\ref{fig:cubic-cases34}(b)). 
Therefore, the proof follows.
\end{proof}

\begin{corollary}
There exists a polynomial-time algorithm that, given a decomposable cubic graph $G$, finds a separating matching of size at least $\mathrm{mms}(G)-2$.
\end{corollary}

\begin{proof}
Let $G$ be a decomposable cubic graph. Compute a maximum matching $M$ of $G$ in polynomial time.
If $G$ contains a bridge, then by Theorem~\ref{th:cubic-bridge}, there exists a separating matching of size at least $|M|-1$. 
Moreover, the proof of Theorem~\ref{th:cubic-bridge} is constructive and shows how to obtain such a separating matching from $M$ by modifying at most two edges. 
This procedure clearly runs in polynomial time.

Otherwise, $G$ is $2$-edge-connected. In this case, by Theorem~\ref{th:cubicn2-2}, there exists a separating matching of size at least $|M|-2$. 
Furthermore, the proof of Theorem~\ref{th:cubicn2-2} is constructive and yields such a separating matching using a polynomial number of local operations.
In both cases, we obtain a separating matching of size at least $\mathrm{mms}(G)-2$ in polynomial time.
\end{proof}

\section{Matching separators in bicubic graphs} \label{sec:bicubic}



We say that a graph $G$ is \textit{bicubic} if it is both bipartite and cubic. In what 
follows, we present a conjecture related to bicubic graphs that admit a perfect matching 
separator. For this, we need some important definitions.
The \textit{disjoint union} of graphs $G$ and $H$, denoted by $G \cup H$, is 
the graph $(V(G) \cup V(H), E(G) \cup E(H))$. Given two cubic graphs $G_1$ and $G_2$, 
and vertices $u \in V(G_1)$ and $v \in V(G_2)$, 
a \textit{star product}~$(G_1, u) * (G_2,v)$ is obtained from $G_1 \cup G_2$ 
by deleting both $u$ and $v$, and by joining the three neighbors of~$u$ in~$G_1$ with the three 
neighbors of~$v$ in~$G_2$ with a matching of size three arbitrarily, see Figure~\ref{fig:starprod}~(c) 
for an example of this operation. A \textit{2-factor} of a graph $G$ is a spanning subgraph $H$ of $G$ 
such that all of its vertices have degree $2$ in $H$. A graph is called \textit{2-factor Hamiltonian} 
if every $2$-factor in the graph induces a Hamiltonian cycle. 

Now, consider $\mathcal{F}$, the class of graphs defined as follows. A graph $G \in \mathcal{F}$ 
if and only if there is a sequence of graphs $G_0 ,G_1, \ldots, G_n$ such that 
\begin{itemize}
    \item $G_0 = H_0$ (see Figure~\ref{fig:starprod}~(a)) or $G_0 = K_{3,3}$,
    \item $G_n = G$,
    \item $G_{i+1}$ is obtained from $G_{i}$ by applying the star product with $H_0$ or $K_{3,3}$.
\end{itemize}
Funk~\cite{Funk2003} studied this class of graphs and conjectured that the 
following holds.

\begin{conjecture}[Funk, 2003] \label{conj:Funk}
    Let $G$ be a $2$-factor Hamiltonian $k$-regular bipartite graph. 
    Then, either $k = 2$ and $G$ is a circuit or $k = 3$ and $G \in \mathcal{F}$. 
\end{conjecture}

\begin{figure}
    \centering
    \includegraphics[width=\linewidth]{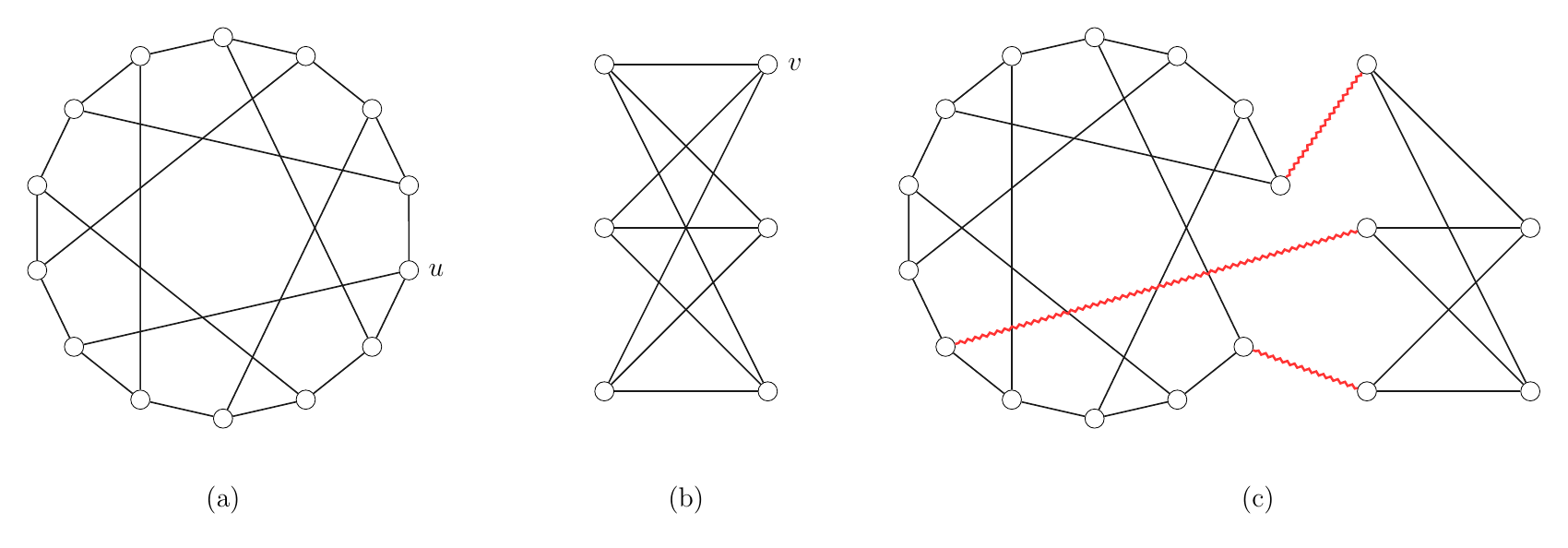}
    \caption{(a) the Headwood graph; (b) $K_{3,3}$; and (c) a star product between the Heawood graph and  $K_{3,3}$.}
    \label{fig:starprod}
\end{figure}

Observe that the complement of a $2$-factor in a cubic graph is a perfect matching.
So, if a cubic graph is $2$-factor Hamiltonian, no perfect matching separates the graph.
Conversely, if a cubic graph is not 2-factor Hamiltonian, there exists a perfect 
matching that separates the graph. Hence, if Conjecture~\ref{conj:Funk} is true, 
it implies that a bicubic graph $G \notin \mathcal{F}$ has a perfect matching separator. 
That is, a matching separator of size~$n/2$, where~$n$ is the number of vertices of $G$. 

Observe that the Heawood graph and $K_{3,3}$ are $2$-factor Hamiltonian, thus, the following result 
of Gorsky et al.~\cite{Gorsky2025} implies that that every graph in $\mathcal{F}$ is $2$-factor 
Hamiltonian. 

\begin{proposition}[Gorsky et al., 2025]
\label{prop:gorksy}
    Let~$G$ be a bicubic graph that can be represented as the star product
    $(G_1,u) * (G_2,v)$ of two bicubic graphs~$G_1$ and~$G_2$.
    Then~$G$ is 2-factor Hamiltonian if and only if~$G_1$ and~$G_2$ are 2-factor Hamiltonian.
\end{proposition}

Although, every graph in $\mathcal{F}$ do not have a perfect matching separator, we will show 
that they have, unless for four graphs, a matching separator with size exactly one less than a 
perfect matching. An \textit{almost 2-factor} in a bicubic graph~$G$ is a disconnected spanning 
subgraph with exactly two vertices of degree $3$, and the rest of vertices of degree $2$. 
Note that $G$ has an almost $2$-factor $H$ if and only if~$G-E(H)$ induces a matching separator 
of size~$|V(G)|/2 -1$ in $G$. The following result tells us that the star product operation 
preserves the existence of an almost $2$-factor.

    \begin{figure}[H]
        \centering
        \includegraphics[width=0.55\linewidth]{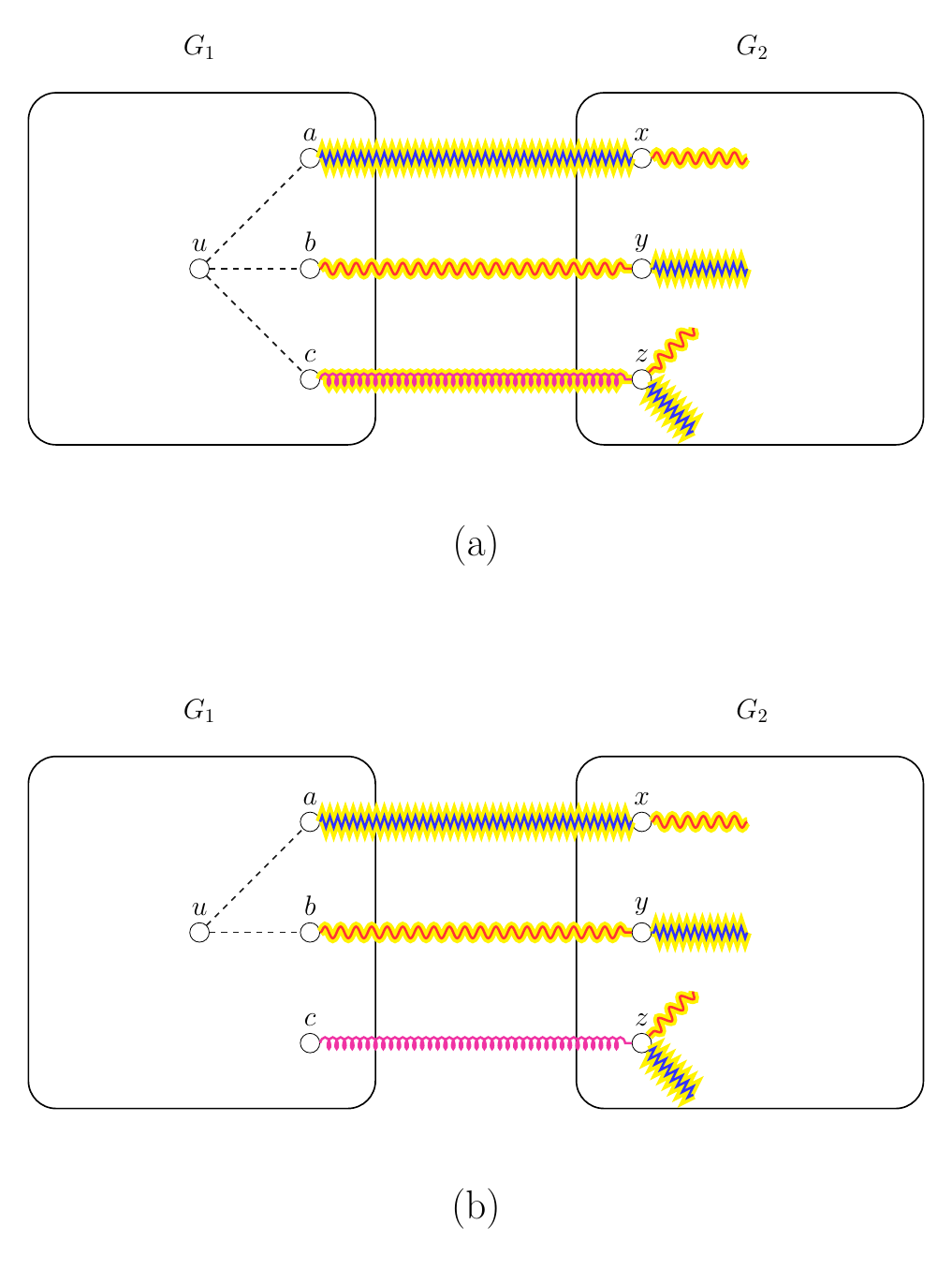}
        \caption{(a) Vertex $u$ has degree $3$ in $F_1$; and (b)~vertex $u$ has degree $2$ in $F_1$. 
        We depict by dashed lines the edges incident to $u$ in $F_1$. Moreover, highlighted lines 
        indicate edges that belong to $F$. Finally, we depict by different kind of lines the colors 
        in $\chi$.}
        \label{fig:almost2factor}
    \end{figure}

\begin{lemma}\label{lemma:star}
Let $G_1$ and $G_2$ be bicubic graphs such that $G_1$ contains an almost $2$-factor, 
and let $u \in V(G_1)$ and $v \in V(G_2)$. If $G = (G_1,u) * (G_2,v)$, then $G$ 
has an almost $2$-factor.
\end{lemma}
\begin{proof}
    Let~$N_{G_1}(u)=\{a,b,c\}$ and~$N_{G_2}(v)=\{x,y,z\}$. Moreover, without loss of generality, 
    suppose that~$ax, by, cz \in E(G)$. In particular, consider the set of edges $M= \{ax, by, cz\}$ 
    as an edge cut that separates $G_1 - u$ from $G_2 - v$ in $G$.
    
    Since $G$ is bipartite and cubic, by K\"onig~\cite{Konig1916} theorem, $G$ has an edge-coloring with~$3$ 
    colors. Let~$\chi = \{X_1,X_2,X_3\}$ be such an edge-coloring of~$G$. In what follows, we show 
    that each edge in $M$ has a distinct color in $\chi$. First, as~$G$ is $2$-factor Hamiltonian, 
    every pair of colors in~$\chi$ induces a Hamiltonian cycle. Thus, for each pair of distinct 
    indices $i, j \in \{1,2,3\}$, we denote by $H_{i,j}$ the Hamiltonian cycle induced by colors $X_i$ 
    and $X_j$. Since any cut intersects every cycle in an even number of edges~\cite{Diestel18} (Theorem 1.9.4), 
    and $H_{i,j}$ is Hamiltonian, we have that $|M \cap H_{i,j}| = 2$. In other words, $M$ contains 
    exactly one edge of each color in $\chi$.

    Without loss of generality, assume that~$$ax \in X_1, by \in X_2, \text{ and } cz \in X_3.$$
    Let~$F_1$ be an almost $2$-factor of~$G_1$, which exists by hypothesis. We conclude the proof 
    by showing how to obtain an almost $2$-factor of $G$ from $F_1$. Let~$F_2$ be the set of edges in~$G_2$ 
    with colors~$1$ and~$2$, that is,~$F_2=(X_1 \cup X_2) \cap E(G_2)$. In what follows, we consider two cases 
    depending on the degree of~$u$ in~$F_1$.

    First, suppose that~$d_{F_1}(u) = 3$. We depict this case in Figure \ref{fig:almost2factor}~(a).
    Let~$$F=(F_1 \cup F_2)-ua-ub-uc+ax+by+cz.$$
    Observe that~$d_F(x)=d_F(y)=2$ and~$d_F(z)=3$. Moreover,~$d_F(p)=d_{F_1}(p)$ 
    for any vertex~$p \in \{a,b,c\}$. Thus, the number of vertices of degree $3$ in~$G[F]$ is exactly~$2$.
    Furthermore, as~$F_1$ is disconnected,~$F$ is also disconnected.
    Therefore, $F$ is an almost 
    $2$-factor of $G$. 
    
   Finally, suppose that $d_{F_1}(u)=2$. We depict this case in Figure \ref{fig:almost2factor}~(b).
   Without loss of generality, suppose that $ua, ub \in F_1$ (otherwise, we relabel the vertices accordingly). 
   Let $$F=(F_1 \cup F_2)-ua-ub+ax+by.$$ 
   Note that~$d_F(p)=d_{F_1}(p)$ for any~$p \in \{a,b,c\}$, and $d_F(p)=d_{F_2}(p)$ 
   for $p \in \{x,y,z\}$. Thus, the number of vertices of degree~$3$ in~$G[F]$ is 
   exactly~$2$. Also, as~$F_1$ is disconnected,~$F$ is disconnected. Therefore, the result 
   follows.
\end{proof}

In what follows, we show that every graph $G \in \mathcal{F}$, unless for four special graphs, 
contains a separating matching of size $|V(G)|/2 - 1$ and, thus, contains an almost $2$-factor. 
We show in Figure~\ref{fig:F} these four special graphs.

\begin{figure}[H]
    \centering
    \includegraphics[width=0.8\linewidth]{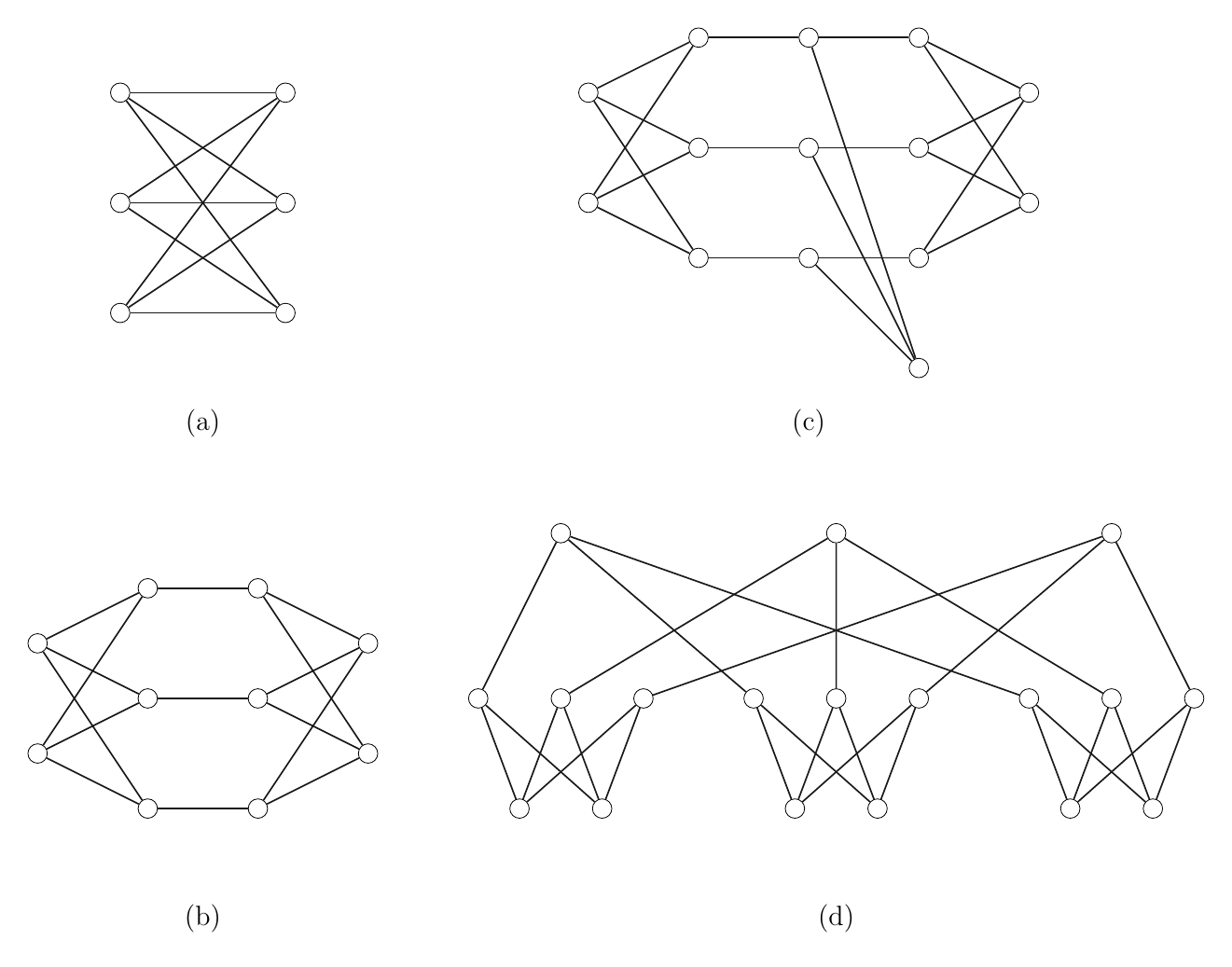}
    \caption{The graphs in $\mathcal{F}$ that do not contain an almost $2$-factor: 
    (a)~$\overline{F}_0$; (b)~$\overline{F}_1$; (c)~$\overline{F}_2$ ; and (d)~$\overline{F}_3$.}
    \label{fig:F}
\end{figure}


\bicubic*

\begin{proof}
    Let $G$ be a graph in $\mathcal{F}$. First, let us recall that every graph in $\mathcal{F}$ 
    is obtained by, starting from $H_0$ or $K_{3,3}$, and by repeatedly applying the star operation 
    with $H_0$ or $K_{3,3}$. Observe that the Headwood graph $H_0$ has a separating matching of 
    size $|V(H_0)|/2 - 1 = 6$ as shown by Figure~\ref{fig:starfam1}~(a). 
    Thus, by Lemma~\ref{lemma:star}, if $G = H_0$ or it can be obtained by applying the star product 
    with $H_0$ at some point, then $G$ contains an almost $2$-factor. Therefore, in what follows, 
    we focus on the graphs in $\mathcal{F}$ that are obtained by starting with $K_{3,3}$ and by 
    repeatedly applying (zero or more times) the star product with $K_{3,3}$.

    \begin{figure}
        \centering
        \includegraphics[width=0.8\linewidth]{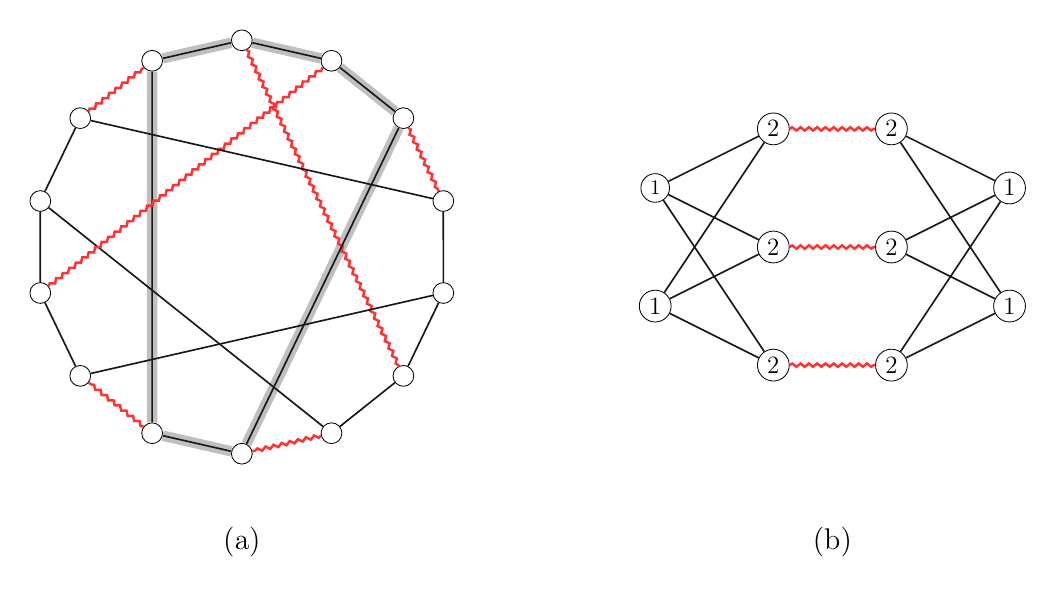}
        \caption{(a) A maximum separating matching of $H_0$; and (b)~a maximum separating 
        matching of $\overline{F}_1$. The curvy edges indicate the edges in the matching.}
        \label{fig:starfam1}
    \end{figure}

    By Theorem~\ref{thm:multi}, $K_{3,3}$ has no separating matching, so it has no 
    almost $2$-factor. By symmetry, there is a single graph (up to isomorphism) that 
    arises from the star product $(K_{3,3}, u) * (K_{3,3}, v)$. We denote this graph 
    by $\overline{F}_1$ (see Figure~\ref{fig:starfam1}~(b)). Now, we show that any separating 
    matching of $\overline{F}_1$ has size $3$. Note that $\overline{F}_1$ is the disjoint union 
    of two $K_{2,3}$ joined by a matching, say $M$, of size $3$. Let $A$ and $B$ be the subgraphs 
    of $\overline{F}_1$ isomorphic to $K_{2,3}$. Let $S$ be a separating matching of~$\overline{F}_1$. 
    By Theorem~\ref{thm:nondecomp}, $K_{2,3}$ is not decomposable ($K_{2,3} = \overline{D_2}$) and, 
    thus, $A \setminus S$ and $B \setminus S$ are connected subgraphs of $\overline{F}_1 \setminus M$. 
    This implies that $M \subseteq S$, otherwise, $F_1 \setminus S$ {\color{red} $\overline{F}_1$} would be connected. Thus  
    $|S| \geq 3$. On the other hand, as $M \subseteq S$, we have that~$E(A) \cap S = \emptyset$ 
    and $E(B) \cap S = \emptyset$. Therefore, $|S| = 3$. 
    
    Now, we focus on the star product $(\overline{F}_1, u) * (K_{3,3},v)$. By symmetry, this product 
    generates two different graphs (up to isomorphism) depending on the choice for $u$. We labeled
    with~$1$ and $2$ the vertices in Figure~\ref{fig:starfam1}~(b). Two vertices with the same label 
    indicate that $(\overline{F}_1, u) * (K_{3,3},v)$ results in the same graph. Let $\overline{F}_2$ 
    and $H_1$ be the two resulting graphs from $(\overline{F}_1, u) * (K_{3,3},v)$. We show in 
    Figure~\ref{fig:starfam2} these graphs. Observe that $H_1$ has a separating matching of size $|V(H_1)|/2 - 1$. 
    In what follows, we show that every separating matching of~$\overline{F}_2$ has size at 
    most $|V(\overline{F}_2)|/2 - 2 = 5$. 
    
    Let $A$ and $B$ be the subgraphs of $\overline{F}_2$ that are isomorphic to $K_{2,3}$. Moreover, 
    let $C = V(\overline{F}_2) \setminus (V(A) \cup V(B))$, and let $M_A$ (resp., $M_B$) be 
    the subset of edges in $\overline{F}_2$ with one end in $A$ (resp., $B$) and the other end in $C$. 
    Let $S$ be any separating matching of $\overline{F}_2$. We will show that 
    either $M_A \subseteq S$ or $M_B \subseteq S$. Note that, this implies that $|S| \leq 5$. 
    Consider that $C = \{a,b,c,d\}$ such that $a, b$ and $c$ are the ends of the edges in $M_A$ 
    that are not in $A$. By contradiction, suppose that $M_A \setminus S \neq \emptyset$ and 
    $M_B \setminus S \neq \emptyset$. We distinguish two cases depending on whether $d$ is matched 
    by $S$ to another vertex in $C$. First, without loss of generality, suppose that $d$ is matched 
    to $a$ by $S$. Then, the other to edges incident to $a$ belong to $\overline{F}_2 - S$. Since $A - S$ 
    and $B - S$ are connected, then $\overline{F}_2 - S$ is connected, a contradiction. Now, suppose that $d$ 
    is not matched by $S$. As $M_A \setminus S \neq \emptyset$ and $M_B \setminus S \neq \emptyset$, 
    let $e \in M_A \setminus S$ and $f \in M_B \setminus S$.
    using $e$,$f$, and an edge incident to $d$, we obtain a path between $A - S$ and $B - S$ in $G - S$ which implies that $S$ is not a separating 
    matching, a contradiction. Therefore, either $M_A \subseteq S$ or $M_B \subseteq S$, and $|S| \leq 5$. 

     \begin{figure}
         \centering
         \includegraphics[width=\linewidth]{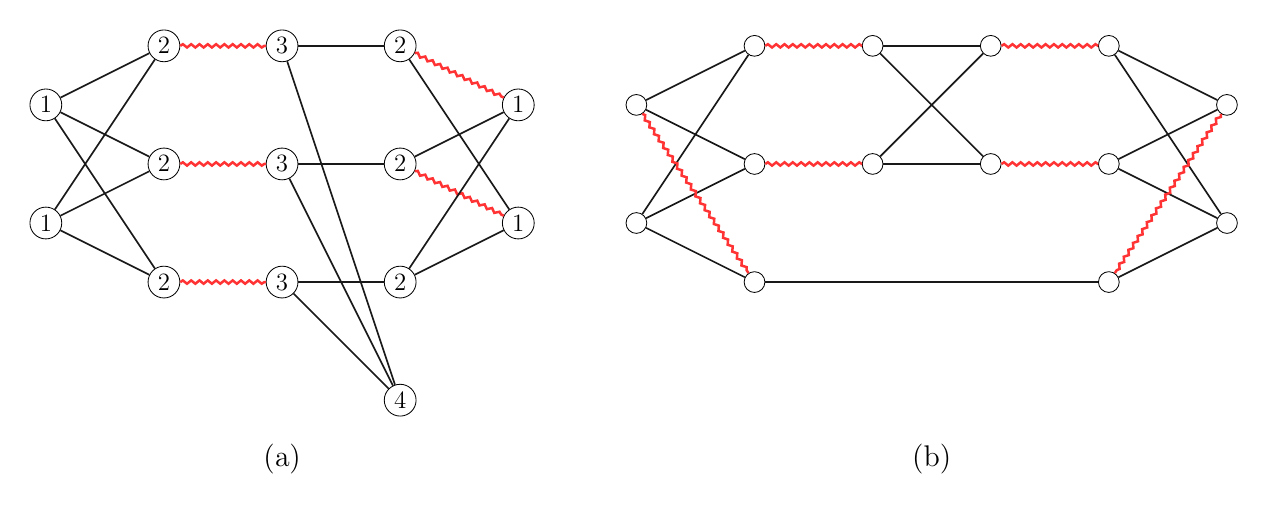}
         \caption{The graphs (a) $\overline{F}_2$ and; (b) $H_1$. The curvy 
         edges induce a maximum separating matching.}
         \label{fig:starfam2}
     \end{figure}

    Now, we consider the star product $(\overline{F}_2, u) * (K_{3,3},v)$. By symmetry, this product 
    generates four different graphs (up to isomorphism) depending on the choice for $u$. Analogously as in 
    the case of $\overline{F}_1$, we labeled with~$1$, $2$, $3$ and $4$ the vertices of $\overline{F}_2$ in Figure~\ref{fig:starfam2}~(a). Two vertices with the same label indicate that $(\overline{F}_2, u) * (K_{3,3},v)$ 
    results in the same graph. 
    \begin{figure}[h]
         \centering
         \includegraphics[width=\linewidth]{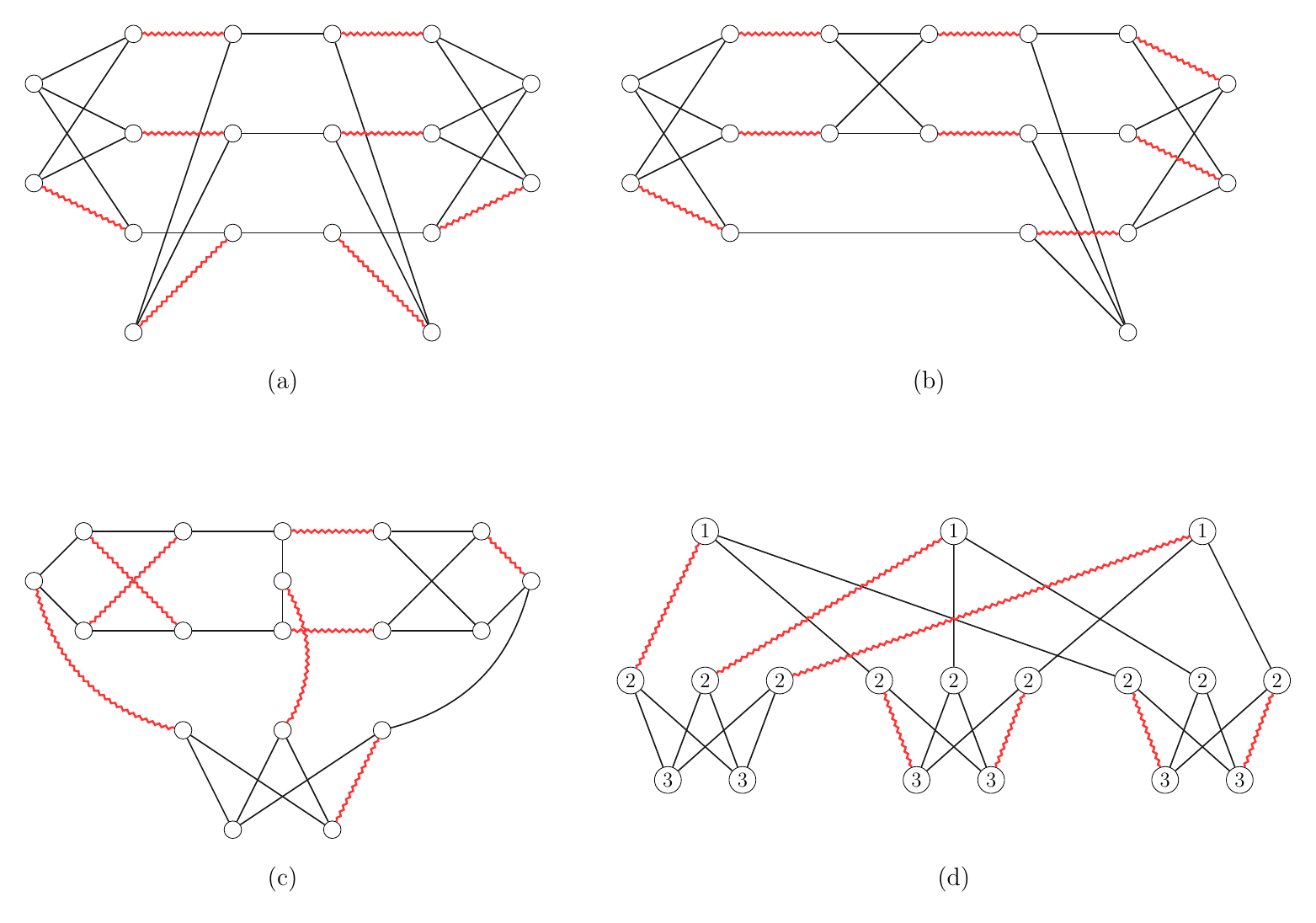}
         \caption{The graphs (a) $H_2$; (b) $H_3$; (c) $H_4$; and (d) $\overline{F}_3$. 
         The curvy edges induce a 
         separating matching.
         }
         \label{fig:starfam3}
     \end{figure}
    Let $H_2$, $H_3$, $H_4$ and $\overline{F}_3$ be the four resulting graphs 
    from $(\overline{F}_2, u) * (K_{3,3},v)$. We show in Figure~\ref{fig:starfam3} these 
    graphs. Note that $H_2$, $H_3$ and $H_4$ have a separating matching of size $|V(H_i)|/2 - 1$, 
    for $i = 2, 3, 4$. Now, we will show that any 
    separating matching of $\overline{F}_3$ has size at most $|V(\overline{F}_3)|/2 - 2 = 7$. 

    Let $A$, $B$ and $C$ be the subgraphs of $\overline{F}_3$ that are isomorphic to $K_{2,3}$. 
    Let $S$ be a separating matching of $\overline{F}_3$, and consider a connected component 
    of $\overline{F}_3 - S$ with the minimum number of vertices, say $K$. We will show 
    that either $V(K) = V(A)$, $V(K) = V(B)$ or~$V(K) = V(C)$. 
    First, by Theorem~\ref{thm:nondecomp}, $K_{2,3}$  is not decomposable. This implies that $V(A)$ 
    is contained in a connected component of $\overline{F}_3 - S$, and the same holds for $V(B)$ 
    and $V(C)$. On the other hand, let $x$, $y$ and $z$ be the vertices of $\overline{F}_3$ 
    not in $V(A) \cup V(B) \cup V(C)$.
    Since $\overline{F}_3$ is cubic, we have that $\overline{F}_3 - S$ has no isolated component
    Therefore, as  $\{x, y, z\}$ is an independent set, $V(A) \subseteq V(K)$, 
    $V(B) \subseteq V(K)$ or~$V(C) \subseteq V(K)$ 
    Without loss of generality, suppose   
    that $V(A) \subseteq V(K)$ (the other cases are symmetric). Furthermore, by contradiction, 
    suppose that $V(A) \neq V(K)$. This implies that $\{x, y, z\} \cap V(K) \neq \emptyset$ because  
    no vertex in $A$ is adjacent to a vertex in $B$ or $C$. Since $\overline{F}_3$ is cubic, 
    we have that $V(A)$ is connected in $\overline{F}_3 - S$ to $B$ or $C$ through a vertex in $\{x, y, z\}$. 
    Thus $K$ contains at least $11$ vertices, but this contradicts the minimally of $K$ as the 
    other connected components of $\overline{F}_3 - S$ have at most $7$ vertices.  Therefore, 
    $V(K) = V(A)$. Since $S$ disconnects $A$ in $\overline{F}_3$, we have that $S$ contains 
    the three edges with one end in $A$ and the other end in $\{x, y, z\}$. Thus, $|S \cap E(B)| \leq 2$ 
    and $|S \cap E(C)| \leq 2$, which implies that $|S| \leq 7$ 

    Finally, we consider the star product $(\overline{F}_3, u) * (K_{3,3},v)$. By symmetry, this product 
    generates three different graphs (up to isomorphism) depending on the choice for $u$. We labeled  
    with~$1$, $2$ and $3$ the vertices of $\overline{F}_3$ in Figure~\ref{fig:starfam2}~(d). Two vertices 
    with the same label indicate that $(\overline{F}_3, u) * (K_{3,3},v)$ results in the same graph. 
    Let $H_5$, $H_6$ and $H_7$ be the graphs that result from this product. In Figure~\ref{fig:starfam4}, 
    we depict a separating matching of size $|V(H_i)|/2 - 1$, for $i = 5, 6, 7 $. Therefore, by Lemma~\ref{lemma:star}, 
    the result follows.
    \begin{figure}[ht]  
        \centering
        \includegraphics[width=0.5\linewidth]{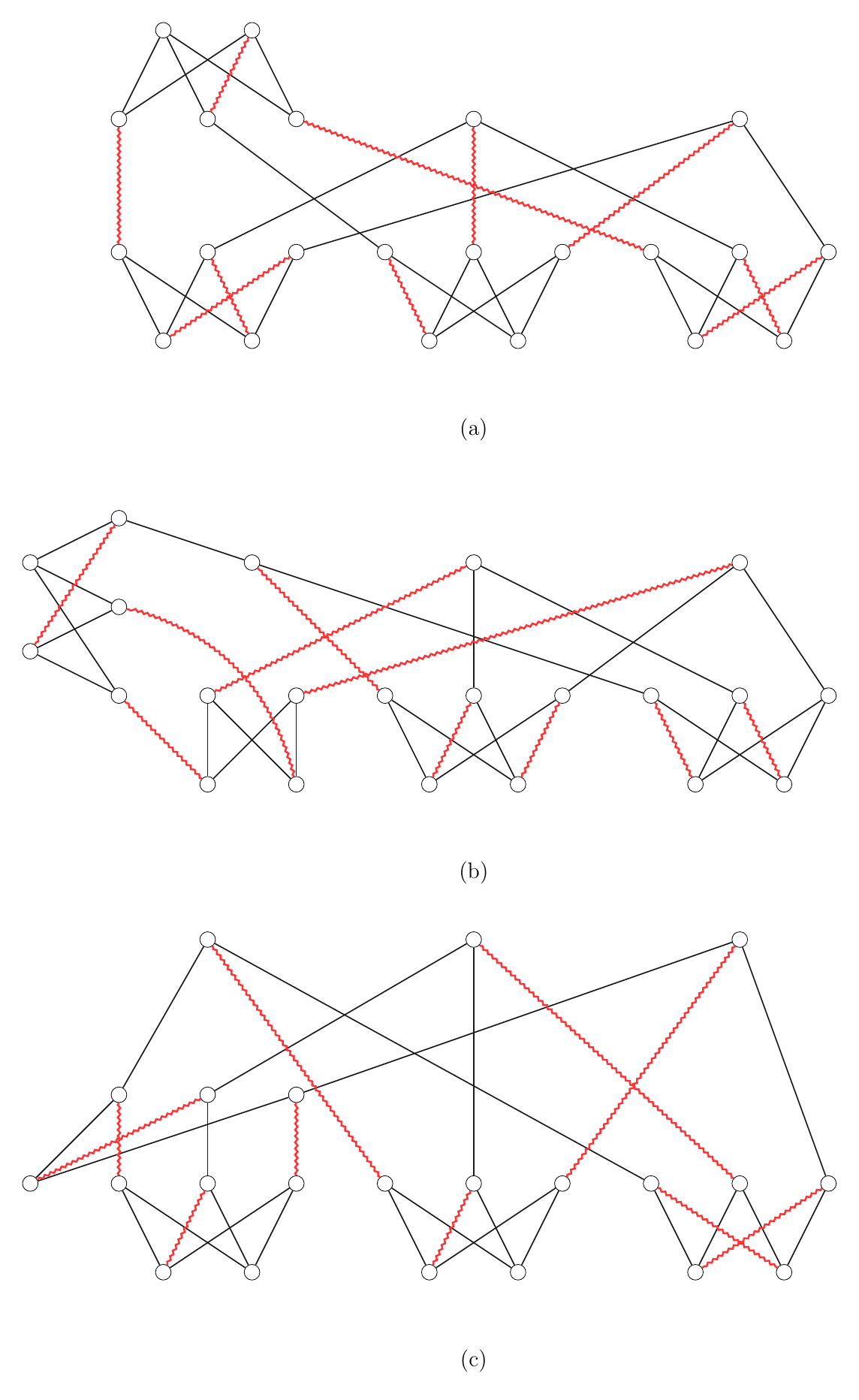}
        \caption{(a) $H_5$; (b) $H_6$; and (c) $H_7$. 
        The curvy edges induce a maximum separating matching.}
        \label{fig:starfam4}
    \end{figure}
\end{proof}

Observe that, if Conjecture~\ref{conj:Funk} is true, we can partition the 
set of bicubic graphs in two sets depending on whether the graph belongs 
to $\mathcal{F}$ or not. Moreover, for every graph $G$ not in $\mathcal{F}$, a 
maximum separating matching has size $|V(G)|/2$. On the other hand, if $G$ belongs 
to $\mathcal{F}$, the size of maximum separating matching is $|V(G)|/2 - 1$, unless 
for $\overline{F}_0$, $\overline{F}_1$, $\overline{F}_2$ and $\overline{F}_3$, in which case, 
the maximum separating matching has size $|V(G)|/2 - 2$. Since these four graphs 
have constant size, the following result follows.


\begin{corollary}
If Conjecture~\ref{conj:Funk} holds, then deciding, for a bicubic graph $G$ and an integer $k$, whether $\mathrm{mms}(G) \ge k$ has the same computational complexity as deciding membership in $\mathcal{F}$.
\end{corollary}

\section{Matching separators in claw-free cubic graphs} \label{sec:claw}

In this section, we show that every claw-free cubic graph has a perfect matching that separates the graph, or equivalently, a disconnecting perfect matching. We begin with a lemma valid for any cubic graph (see also Proposition~4.13 of~\cite{Bouquet2025}). 

\begin{lemma}\label{lemma:cubic-bridge}
Every cubic graph with a bridge and a perfect matching has a disconnecting perfect matching.
\end{lemma}

\begin{proof}
Let $G$ be a cubic graph with a bridge $uv$ and a perfect matching $M$. We may assume that $G$ is connected, as otherwise we apply the argument to each component. If $uv \in M$, then $M$ is a separating matching, and we are done. Otherwise, suppose that one of $u$ and $v$ is not saturated by $M$, say $u$. Let $vx \in M$. Then $M' = (M \setminus \{vx\}) \cup \{uv\}$ is a perfect matching, and since $uv \in M'$, it is separating. Hence, we may assume that both $u$ and $v$ are saturated by $M$. Let $G_u$ be the component of $G - uv$ containing $u$. Since $M$ is perfect, $G_u$ has even order. However, $u$ has degree $2$ in $G_u$, while every other vertex has degree $3$, which is impossible.
\end{proof}

We will use the following characterization due to Oum~\cite{Oum2009}.

\begin{lemma}[Proposition~1 of~\cite{Oum2009}]\label{lemma:claw-free-characterization}
A graph $G$ is cubic, bridgeless, and claw-free if and only if one of the following holds:
\begin{itemize}
    \item $G \simeq K_4$,
    \item $G$ is a ring of diamonds, or
    \item there exists a cubic multigraph $H$ such that $G$ is obtained from $H$ by replacing each edge with a string of diamonds and each vertex with a triangle.
\end{itemize}
\end{lemma}

In this context, a \emph{diamond} is an induced subgraph isomorphic to $K_4 - e$. 
We say that two diamonds are \emph{adjacent} if there exists an edge joining a degree-$2$ vertex of one diamond to a degree-$2$ vertex of the other. 
Contracting each diamond to a single vertex and keeping these adjacencies yields an auxiliary graph; if this graph is a path, we call the corresponding subgraph a \emph{string of diamonds}, and if it is a cycle, we call it a \emph{ring of diamonds} (Figure~\ref{fig:claw-free-a}).

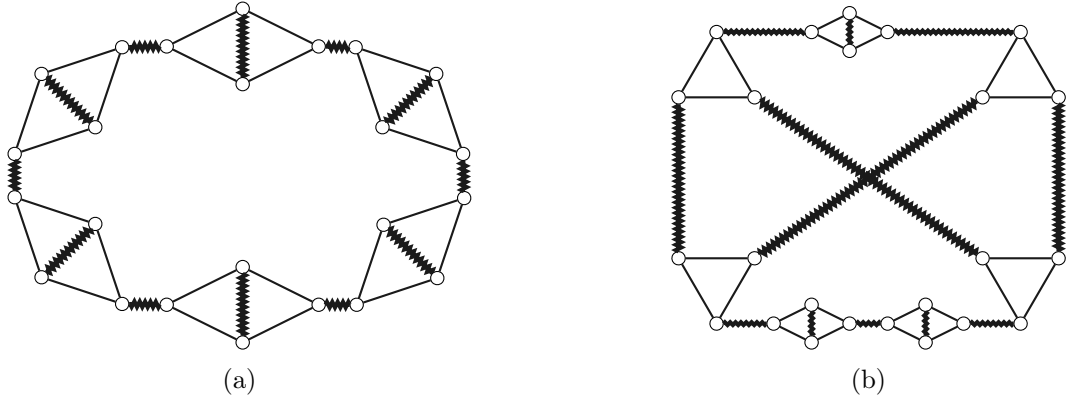
\begin{figure}[H]
\centering

\begin{subfigure}{0.48\textwidth}
\centering
\begin{tikzpicture}[scale=0.5, transform shape]
   
    \begin{scope}[yshift=1cm]
    \node[vertex] (A0) {};
    \node[vertex] (A1) at ($(A0)+(4,0)$) {};
    \node[vertex] (A2) at ($(A0)+(2,1)$) {};
    \node[vertex] (A3) at ($(A0)+(2,-1)$) {};
    \foreach \source/\dest in {A0/A2,A0/A3,A1/A2,A1/A3,A2/A3}
        \path[black edge] (\source) -- (\dest);
    \end{scope}

    \draw[black edge, ultra thick,
      decorate,
      decoration={zigzag, amplitude=1pt, segment length=2pt}]
      (A2) -- (A3);
        
    \begin{scope}[xshift=5cm,yshift=1cm, rotate=45]
    \node[vertex] (B0) {};
    \node[vertex] (B1) at ($(B0)+(4,0)$) {};
    \node[vertex] (B2) at ($(B0)+(2,1)$) {};
    \node[vertex] (B3) at ($(B0)+(2,-1)$) {};
    \foreach \source/\dest in {B0/B2,B0/B3,B1/B2,B1/B3,B2/B3}
        \path[black edge] (\source) -- (\dest);
    \end{scope}

     \draw[black edge, ultra thick,
      decorate,
      decoration={zigzag, amplitude=1pt, segment length=2pt}]
      (B2) -- (B3);

    \begin{scope}[xshift=7.8cm,yshift=5cm, rotate=135]
    \node[vertex] (C0) {};
    \node[vertex] (C1) at ($(C0)+(4,0)$) {};
    \node[vertex] (C2) at ($(C0)+(2,1)$) {};
    \node[vertex] (C3) at ($(C0)+(2,-1)$) {};
    \foreach \source/\dest in {C0/C2,C0/C3,C1/C2,C1/C3,C2/C3}
        \path[black edge] (\source) -- (\dest);
    \end{scope}

    \draw[black edge, ultra thick,
      decorate,
      decoration={zigzag, amplitude=1pt, segment length=2pt}]
      (C2) -- (C3);

    \begin{scope}[yshift=7.85cm]
    \node[vertex] (D0) {};
    \node[vertex] (D1) at ($(D0)+(4,0)$) {};
    \node[vertex] (D2) at ($(D0)+(2,1)$) {};
    \node[vertex] (D3) at ($(D0)+(2,-1)$) {};
    \foreach \source/\dest in {D0/D2,D0/D3,D1/D2,D1/D3,D2/D3}
        \path[black edge] (\source) -- (\dest);
    \end{scope}

    \draw[black edge, ultra thick,
      decorate,
      decoration={zigzag, amplitude=1pt, segment length=2pt}]
      (D2) -- (D3);

    \begin{scope}[xshift=-4cm,yshift=3.85cm, rotate=-45]
    \node[vertex] (F0) {};
    \node[vertex] (F1) at ($(F0)+(4,0)$) {};
    \node[vertex] (F2) at ($(F0)+(2,1)$) {};
    \node[vertex] (F3) at ($(F0)+(2,-1)$) {};
    \foreach \source/\dest in {F0/F2,F0/F3,F1/F2,F1/F3,F2/F3}
        \path[black edge] (\source) -- (\dest);
    \end{scope}

    \draw[black edge, ultra thick,
      decorate,
      decoration={zigzag, amplitude=1pt, segment length=2pt}]
      (F2) -- (F3);

    \begin{scope}[xshift=-4cm,yshift=5cm, rotate=45]
    \node[vertex] (E0) {};
    \node[vertex] (E1) at ($(E0)+(4,0)$) {};
    \node[vertex] (E2) at ($(E0)+(2,1)$) {};
    \node[vertex] (E3) at ($(E0)+(2,-1)$) {};
    \foreach \source/\dest in {E0/E2,E0/E3,E1/E2,E1/E3,E2/E3}
        \path[black edge] (\source) -- (\dest);
    \end{scope}

        \draw[black edge, ultra thick,
      decorate,
      decoration={zigzag, amplitude=1pt, segment length=2pt}]
      (E2) -- (E3);

   \draw[black edge, ultra thick,
      decorate,
      decoration={zigzag, amplitude=1pt, segment length=2pt}]
      (A0) -- (F1);
    \draw[black edge, ultra thick,
      decorate,
      decoration={zigzag, amplitude=1pt, segment length=2pt}] (D0) -- (E1);
    \draw[black edge, ultra thick,
      decorate,
      decoration={zigzag, amplitude=1pt, segment length=2pt}] (E0) -- (F0);
    \draw[black edge, ultra thick,
      decorate,
      decoration={zigzag, amplitude=1pt, segment length=2pt}] (C1) -- (D1);
    \draw[black edge, ultra thick,
      decorate,
      decoration={zigzag, amplitude=1pt, segment length=2pt}] (B1) -- (C0);
    \draw[black edge, ultra thick,
      decorate,
      decoration={zigzag, amplitude=1pt, segment length=2pt}] (A1) -- (B0);

\end{tikzpicture}
\caption{}
\label{fig:claw-free-a}
\end{subfigure}
\hfill
\begin{subfigure}{0.48\textwidth}
\centering
\begin{tikzpicture}[scale=0.5, transform shape]
    \begin{scope}[xshift=-4cm, yshift=1cm, rotate=60]
    \node[vertex] (A0) {};
    \node[vertex] (A1) at ($(A0)+(2,0)$) {};
    \node[vertex] (A2) at ($(A0)+(1,{1*sqrt(3)})$) {};
    \foreach \source/\dest in {A0/A1,A1/A2,A2/A0}
        \path[black edge] (\source) -- (\dest);
    \end{scope}

    \begin{scope}[xshift=-5cm, yshift=7cm, rotate=0]
    \node[vertex] (B0) {};
    \node[vertex] (B1) at ($(B0)+(2,0)$) {};
    \node[vertex] (B2) at ($(B0)+(1,{1*sqrt(3)})$) {};
    \foreach \source/\dest in {B0/B1,B1/B2,B2/B0}
        \path[black edge] (\source) -- (\dest);
    \end{scope}

    \begin{scope}[xshift=4cm, yshift=1cm, rotate=60]
    \node[vertex] (C0) {};
    \node[vertex] (C1) at ($(C0)+(2,0)$) {};
    \node[vertex] (C2) at ($(C0)+(1,{1*sqrt(3)})$) {};
    \foreach \source/\dest in {C0/C1,C1/C2,C2/C0}
        \path[black edge] (\source) -- (\dest);
    \end{scope}

    \begin{scope}[xshift=3cm, yshift=7cm, rotate=0]
    \node[vertex] (D0) {};
    \node[vertex] (D1) at ($(D0)+(2,0)$) {};
    \node[vertex] (D2) at ($(D0)+(1,{1*sqrt(3)})$) {};
    \foreach \source/\dest in {D0/D1,D1/D2,D2/D0}
        \path[black edge] (\source) -- (\dest);
    \end{scope}

\begin{scope}[xshift=-1.5cm, yshift={(7+sqrt(3))*1cm}]

  \node[vertex] (W0) {};
    \node[vertex] (W1) at ($(W0)+(2,0)$) {};
    \node[vertex] (W2) at ($(W0)+(1,0.5)$) {};
    \node[vertex] (W3) at ($(W0)+(1,-0.5)$) {};
    \foreach \source/\dest in {W0/W2,W0/W3,W1/W2,W1/W3,W2/W3}
        \path[black edge] (\source) -- (\dest);
\end{scope}
    
     \begin{scope}[xshift=-2.5cm, yshift=1cm]
    \node[vertex] (Y0) {};
    \node[vertex] (Y1) at ($(Y0)+(2,0)$) {};
    \node[vertex] (Y2) at ($(Y0)+(1,0.5)$) {};
    \node[vertex] (Y3) at ($(Y0)+(1,-0.5)$) {};
    \foreach \source/\dest in {Y0/Y2,Y0/Y3,Y1/Y2,Y1/Y3,Y2/Y3}
        \path[black edge] (\source) -- (\dest);
    \end{scope}

    \begin{scope}[xshift=0.5cm, yshift=1cm]
    \node[vertex] (Z0) {};
    \node[vertex] (Z1) at ($(Z0)+(2,0)$) {};
    \node[vertex] (Z2) at ($(Z0)+(1,0.5)$) {};
    \node[vertex] (Z3) at ($(Z0)+(1,-0.5)$) {};
    \foreach \source/\dest in {Z0/Z2,Z0/Z3,Z1/Z2,Z1/Z3,Z2/Z3}
        \path[black edge] (\source) -- (\dest);
    \end{scope}

    \draw[black edge, ultra thick,
      decorate,
      decoration={zigzag, amplitude=1pt, segment length=2pt}] (A1) -- (D0);
    \draw[black edge, ultra thick,
      decorate,
      decoration={zigzag, amplitude=1pt, segment length=2pt}] (B1) -- (C2);
    \draw[black edge, ultra thick,
      decorate,
      decoration={zigzag, amplitude=1pt, segment length=2pt}] (A2) -- (B0);
    \draw[black edge, ultra thick,
      decorate,
      decoration={zigzag, amplitude=1pt, segment length=2pt}] (D1) -- (C1);

     \draw[black edge, ultra thick,
      decorate,
      decoration={zigzag, amplitude=0.5pt, segment length=2pt}] (W2) -- (W3);
      
     \draw[black edge, ultra thick,
      decorate,
      decoration={zigzag, amplitude=0.5pt, segment length=2pt}] (Y2) -- (Y3);

      \draw[black edge, ultra thick,
      decorate,
      decoration={zigzag, amplitude=0.5pt, segment length=2pt}] (Z2) -- (Z3);

    \draw[black edge, ultra thick,
      decorate,
      decoration={zigzag, amplitude=0.5pt, segment length=2pt}] (A0) -- (Y0);
    \draw[black edge, ultra thick,
      decorate,
      decoration={zigzag, amplitude=0.5pt, segment length=2pt}] (Y1) -- (Z0);
   \draw[black edge, ultra thick,
      decorate,
      decoration={zigzag, amplitude=0.5pt, segment length=2pt}] (Z1) -- (C0);

      \draw[black edge, ultra thick,
      decorate,
      decoration={zigzag, amplitude=0.5pt, segment length=2pt}] (W1) -- (D2);

      \draw[black edge, ultra thick,
      decorate,
      decoration={zigzag, amplitude=0.5pt, segment length=2pt}] (W0) -- (B2);

\end{tikzpicture}

\caption{}
\label{fig:claw-free-b}
\end{subfigure}

\caption{(a) A ring of diamonds. (b) A claw-free cubic graph obtained from $H = K_4$. In both cases, darker edges indicate a separating perfect matching.}
\label{fig:claw-free}
\end{figure}


\clawfree*
\begin{proof}
If $G$ has a bridge, then the result follows from Lemma~\ref{lemma:cubic-bridge}. 
Otherwise, by Lemma~\ref{lemma:claw-free-characterization}, $G$ is either a ring of diamonds or is obtained from a cubic multigraph $H$.

If $G$ is a ring of diamonds, then in each diamond we select the unique edge joining its two vertices of degree $3$. 
For every pair of consecutive diamonds, we choose exactly one of the two edges joining their degree-$2$ vertices. 
It is straightforward to verify that these edges form a perfect matching of $G$. 
Moreover, removing them disconnects each diamond from its neighbours, and hence the matching is separating (see Figure~\ref{fig:claw-free-a}).

Otherwise, let $H'$ be the graph obtained from $H$ by replacing each vertex with a triangle. 
We claim that $E(H)$ is a separating matching in $H'$. Indeed, every edge of $H$ corresponds to an edge in $H'$ joining distinct triangles, so $E(H)$ is a matching. 
Moreover, removing $E(H)$ separates every triangle from the rest of the graph.

Finally, we argue by induction on the number of edges of $H$ that are replaced by strings of diamonds. 
If no edge is replaced, then $G = H'$ and the result follows from the previous paragraph. 
Otherwise, let $uv$ be an edge of $H$ that is replaced by a string of diamonds. 
By the inductive hypothesis, there exists a separating perfect matching $M$ in the graph before replacing $uv$, and in particular $uv \in M$. 
A string of diamonds admits a perfect matching that covers all its internal vertices and matches its two extremities exactly as $uv$ did. 
Replacing $uv$ in $M$ with this matching yields a perfect matching of $G$. 
Since the replacement only affects vertices inside the string, and the original matching was separating, the resulting matching is also separating (see Figure~\ref{fig:claw-free-b}).
\end{proof}



\bibliographystyle{plain}
\bibliography{bibliography}

@book{Bondy2008,
  author    = {J. A. Bondy and U. S. R. Murty},
  title     = {Graph Theory},
  series    = {Graduate Texts in Mathematics},
  volume    = {244},
  publisher = {Springer},
  year      = {2008}
}

@article {Chvatal1984,
    AUTHOR = {Chv\'atal, V.},
     TITLE = {Recognizing decomposable graphs},
   JOURNAL = {J. Graph Theory},
  FJOURNAL = {Journal of Graph Theory},
    VOLUME = {8},
      YEAR = {1984},
    NUMBER = {1},
     PAGES = {51--53},
      ISSN = {0364-9024,1097-0118},
   MRCLASS = {05C15},
  MRNUMBER = {732017},
MRREVIEWER = {Fan\ Chung\ Graham},
}

@book {Diestel18,
    AUTHOR = {Diestel, R.},
     TITLE = {Graph theory},
    SERIES = {Graduate Texts in Mathematics},
    VOLUME = {173},
   EDITION = {Fifth},
 PUBLISHER = {Springer, Berlin},
      YEAR = {2018},
     PAGES = {xviii+428},
      ISBN = {978-3-662-57560-4; 978-3-662-53621-6},
}

@article {Graham1970,
    AUTHOR = {Graham, R.},
     TITLE = {On primitive graphs and optimal vertex assignments},
   JOURNAL = {Ann. New York Acad. Sci.},
  FJOURNAL = {Annals of the New York Academy of Sciences},
    VOLUME = {175},
      YEAR = {1970},
     PAGES = {170--186},
      ISSN = {0077-8923,1749-6632},
   MRCLASS = {05.40},
  MRNUMBER = {269533},
MRREVIEWER = {L.\ H.\ Harper},
}

@article {Diwan02,
    AUTHOR = {Diwan, A.},
     TITLE = {Disconnected 2-factors in planar cubic bridgeless graphs},
   JOURNAL = {J. Combin. Theory Ser. B},
  FJOURNAL = {Journal of Combinatorial Theory. Series B},
    VOLUME = {84},
      YEAR = {2002},
    NUMBER = {2},
     PAGES = {249--259},
      ISSN = {0095-8956,1096-0902},
   MRCLASS = {05C70},
  MRNUMBER = {1889257},
MRREVIEWER = {Martin\ Ba\v ca},
}

@article {Funk2003,
    AUTHOR = {Funk, M. and Jackson, B. and Labbate, D. and Sheehan, J.},
     TITLE = {2-factor {H}amiltonian graphs},
   JOURNAL = {J. Combin. Theory Ser. B},
  FJOURNAL = {Journal of Combinatorial Theory. Series B},
    VOLUME = {87},
      YEAR = {2003},
    NUMBER = {1},
     PAGES = {138--144},
      ISSN = {0095-8956,1096-0902},
   MRCLASS = {05C45 (05C70)},
  MRNUMBER = {1967885},
MRREVIEWER = {Akira\ Saito},
}

@misc{le2023maximizingmatchingcuts,
      title={Maximizing Matching Cuts}, 
      author={Van Bang Le and Felicia Lucke and Daniël Paulusma and Bernard Ries},
      year={2023},
      eprint={2312.12960},
      archivePrefix={arXiv},
      primaryClass={math.CO},
      url={https://arxiv.org/abs/2312.12960}, 
}

@article {Gorsky2025,
    AUTHOR = {Gorsky, M. and Johanni, T. and Wiederrecht,
              S.},
     TITLE = {A note on the 2-factor {H}amiltonicity {C}onjecture},
   JOURNAL = {Discrete Math.},
  FJOURNAL = {Discrete Mathematics},
    VOLUME = {348},
      YEAR = {2025},
    NUMBER = {6},
     PAGES = {Paper No. 114442},
      ISSN = {0012-365X,1872-681X},
   MRCLASS = {05C70 (05C45)},
  MRNUMBER = {4867192},
}

@article {Le2019,
    AUTHOR = {Le, H. and Le, V. B.},
     TITLE = {A complexity dichotomy for matching cut in (bipartite) graphs
              of fixed diameter},
   JOURNAL = {Theoret. Comput. Sci.},
  FJOURNAL = {Theoretical Computer Science},
    VOLUME = {770},
      YEAR = {2019},
     PAGES = {69--78},
      ISSN = {0304-3975,1879-2294},
   MRCLASS = {68R10 (68Q17 68Q25)},
  MRNUMBER = {3944735},
}

@incollection {Le2023,
    AUTHOR = {Le, H. and Le, V. B.},
     TITLE = {Complexity results for matching cut problems in graphs without
              long induced paths},
 BOOKTITLE = {Graph-theoretic concepts in computer science},
    SERIES = {Lecture Notes in Comput. Sci.},
    VOLUME = {14093},
     PAGES = {417--431},
 PUBLISHER = {Springer, Cham},
      YEAR = {2023},
      ISBN = {978-3-031-43379-5; 978-3-031-43380-1},
   MRCLASS = {05C70 (05C40)},
  MRNUMBER = {4657731},
}

@article {LuckePaulusmaRies2023,
    AUTHOR = {Lucke, F. and Paulusma, D. and Ries, B.},
     TITLE = {Dichotomies for maximum matching cut: {$H$}-freeness, bounded
              diameter, bounded radius},
   JOURNAL = {Theoret. Comput. Sci.},
  FJOURNAL = {Theoretical Computer Science},
    VOLUME = {1017},
      YEAR = {2024},
     PAGES = {Paper No. 114795, 18},
      ISSN = {0304-3975,1879-2294},
   MRCLASS = {68R10},
  MRNUMBER = {4790920},
MRREVIEWER = {Ton\ Kloks},
}

@article {LuckePaulusmaRies2022,
    AUTHOR = {Lucke, F. and Paulusma, D. and Ries, B.},
     TITLE = {On the complexity of matching cut for graphs of bounded radius
              and {$H$}-free graphs},
   JOURNAL = {Theoret. Comput. Sci.},
  FJOURNAL = {Theoretical Computer Science},
    VOLUME = {936},
      YEAR = {2022},
     PAGES = {33--42},
      ISSN = {0304-3975,1879-2294},
   MRCLASS = {68R10},
  MRNUMBER = {4498350},
MRREVIEWER = {Yuxing\ Yang},
}

@article {Moshi89,
    AUTHOR = {Moshi, A.},
     TITLE = {Matching cutsets in graphs},
   JOURNAL = {J. Graph Theory},
  FJOURNAL = {Journal of Graph Theory},
    VOLUME = {13},
      YEAR = {1989},
    NUMBER = {5},
     PAGES = {527--536},
      ISSN = {0364-9024,1097-0118},
   MRCLASS = {05C70 (68R10)},
  MRNUMBER = {1016273},
MRREVIEWER = {Marek\ Kubale},
       DOI = {10.1002/jgt.3190130502},
       URL = {https://doi.org/10.1002/jgt.3190130502},
}

@article {Konig1916,
    AUTHOR = {K\"onig, D.},
     TITLE = {\"Uber {G}raphen und ihre {A}nwendung auf
              {D}eterminantentheorie und {M}engenlehre},
   JOURNAL = {Math. Ann.},
  FJOURNAL = {Mathematische Annalen},
    VOLUME = {77},
      YEAR = {1916},
    NUMBER = {4},
     PAGES = {453--465},
      ISSN = {0025-5831,1432-1807},
   MRCLASS = {99-04},
  MRNUMBER = {1511872},
}

@incollection {Patrignani2001,
    AUTHOR = {Patrignani, M. and Pizzonia, M.},
     TITLE = {The complexity of the matching-cut problem},
 BOOKTITLE = {Graph-theoretic concepts in computer science ({B}oltenhagen,
              2001)},
    SERIES = {Lecture Notes in Comput. Sci.},
    VOLUME = {2204},
     PAGES = {284--295},
 PUBLISHER = {Springer, Berlin},
      YEAR = {2001},
      ISBN = {3-540-42707-4},
   MRCLASS = {68R10 (05C85 68Q25)},
  MRNUMBER = {1905640},
}

@article {Petersen1900,
    AUTHOR = {Petersen, J.},
     TITLE = {Die {T}heorie der regul\"aren graphs},
   JOURNAL = {Acta Math.},
  FJOURNAL = {Acta Mathematica},
    VOLUME = {15},
      YEAR = {1891},
    NUMBER = {1},
     PAGES = {193--220},
      ISSN = {0001-5962,1871-2509},
   MRCLASS = {99-04},
  MRNUMBER = {1554815},
}

@article {Bouquet2025,
    AUTHOR = {Bouquet, V. and Picouleau, C.},
     TITLE = {The complexity of the perfect matching-cut problem},
   JOURNAL = {J. Graph Theory},
  FJOURNAL = {Journal of Graph Theory},
    VOLUME = {108},
      YEAR = {2025},
    NUMBER = {3},
     PAGES = {432--462},
      ISSN = {0364-9024,1097-0118},
   MRCLASS = {68R10 (05C70)},
  MRNUMBER = {4850657},
MRREVIEWER = {Dinabandhu\ Pradhan},
}

@article {Oum2009,
    AUTHOR = {Oum, S.},
     TITLE = {Perfect matchings in claw-free cubic graphs},
   JOURNAL = {Electron. J. Combin.},
  FJOURNAL = {Electronic Journal of Combinatorics},
    VOLUME = {18},
      YEAR = {2011},
    NUMBER = {1},
     PAGES = {Paper 62, 6},
      ISSN = {1077-8926},
   MRCLASS = {05C70},
  MRNUMBER = {2788679},
MRREVIEWER = {Zan-Bo\ Zhang},
}

@article{Sumner1974,
  author = {Sumner, D. P.},
  title = {Graphs with 1-factors},
  journal = {Proceedings of the American Mathematical Society},
  volume = {42},
  number = {1},
  year = {1974},
  pages = {8--12}
}

@article{LasVergnas1975,
  author  = {Las Vergnas, M.},
  title   = {A note on matchings in graphs},
  journal = {Cahiers du Centre d'Études de Recherche Opérationnelle},
  volume  = {17},
  number  = {2-3-4},
  pages   = {257--260},
  year    = {1975},
  note    = {Colloque sur la Th\'eorie des Graphes (Paris, 1974)}
}

\end{document}